\documentclass[reqno]{article}

\usepackage{amsmath,amssymb,amsfonts,amsthm}
\usepackage{psfrag}
\usepackage{pstool}
\usepackage{cases}
\newcommand{\Eqref}[1]{\eqref{#1}}

\newcommand{\Secref}[1]{Section~\ref{#1}}
\newcommand{\Thmref}[1]{Theorem~\ref{#1}}
\newcommand{\Corref}[1]{Corollary~\ref{#1}}
\newcommand{\Lemref}[1]{Lemma~\ref{#1}}
\newcommand{\Propref}[1]{Prop.~\ref{#1}}
\newcommand{\Defref}[1]{Definition~\ref{#1}}
\newcommand{\Remref}[1]{Remark~\ref{#1}}

\newcommand{\Stepref}[1]{Step~\ref{#1}}

\newcommand{\R}{\mathbb{R}} 
\newcommand{\Z}{\mathbb{Z}} 
\newcommand{\Prob}{\mathbb{P}} 
\newcommand{\C}{\mathbb{C}}
\newcommand{\Ord}{\mathcal{O}} 
\newcommand{\w}{\omega}

\newcommand{\Norm}[2]{\left\lVert{#1}\right\rVert_{#2}} 
\newcommand{\AND}{\textrm{ and }}

\newcommand{\Var}{\textrm{Var}}

\newcommand{\Id}{\mathbf{1}}
\newcommand{\Cv}[1]{\mathcal{C}_{#1}}
\newcommand{\Cs}[1]{\mathcal{D}_{#1}}
\newcommand{\Cw}[1]{\ell_{#1}} 
\newcommand{\I}{{\rm i}}

\newcommand{\EE}{\ensuremath{\mathbb{E}}}

\renewcommand{\Re}{\operatorname{Re}}
\renewcommand{\Im}{\operatorname{Im}}

\DeclareMathOperator{\Ai}{Ai}

\def\arg{\operatorname{arg}}

\newcommand{\Kloggamma}{K_u^{N}} 
\newcommand{\Ksteepestdescent}{K^{N}} 
\newcommand{\Kairy}{K_{\Ai}} 
\newcommand{\zcrit}{z_{\rm crit}}
\newcommand{\Res}{\operatorname{Res}}
\newcommand{\loggamma}{log-gamma}
\newcommand{\expgamma}{exp-gamma}
\newcommand{\partfunc}{Z_{\beta}^N}
\newcommand{\Kbcfv}{K^N_{u,\tau}} 

\newcommand{\E}{\mathbb{E}\,}

\theoremstyle{plain}
\newtheorem{theorem}{Theorem}[section]
\newtheorem*{theorem*}{Theorem}
\newtheorem{lemma}[theorem]{Lemma}
\newtheorem*{lemma*}{Lemma}
\newtheorem{cor}[theorem]{Corollary}

\newtheorem{prop}[theorem]{Proposition}
\newtheorem*{prop*}{Proposition}

\theoremstyle{definition}
\newtheorem{define}[theorem]{Definition}

\newtheorem*{example*}{Example}

\theoremstyle{remark}
\newtheorem{remark}{Remark}
\newtheorem*{remark*}{Remark}

\makeatletter
\def\@wraptoccontribs#1#2{}
\makeatother

\usepackage[numbers,sort&compress]{natbib}
\usepackage[dvipsnames]{xcolor}

\usepackage{todonotes}

\usepackage[normalem]{ulem} 

\usepackage[textwidth=6in]{geometry}

\def\wt{\xi} 
\def\path{\mathbf{x}} 

\usepackage{versions}

\includeversion{note}       
\includeversion{longcalc}   
\includeversion{newnotes}   

\markversion{note}
\markversion{newnotes}
\markversion{longcalc}
\markversion{rambling}
\markversion{oldnotes}

\excludeversion{note}
\excludeversion{longcalc}
\excludeversion{rambling}   
\excludeversion{oldnotes}   



\usepackage[nonotes]{optional}

\title{Tracy-Widom fluctuations for perturbations of the log-gamma polymer in intermediate disorder}
\date{}

\author{Arjun Krishnan \thanks{Department of Mathematics, University of Rochester.}
\and Jeremy Quastel \thanks{Department of Mathematics, University of Toronto.}}
\usepackage[urlcolor=blue,colorlinks=true,linkcolor=blue,citecolor=blue]{hyperref}

\renewcommand{\i}{\ell}

\begin{document}
\maketitle

\begin{abstract}
    The free-energy fluctuations of the discrete directed polymer in $1+1$ dimensions is conjecturally in the Tracy-Widom universality class at all finite temperatures and in the intermediate disorder regime. Sepp\"al\"ainen's log-gamma polymer was proven to have GUE Tracy-Widom fluctuations in a restricted temperature range by~\citet{MR3116323}. We remove this restriction, and extend this result into the intermediate disorder regime. This result also identifies the scale of fluctuations of the \loggamma{} polymer in the intermediate disorder regime, and thus verifies a conjecture of \citet{MR3189070}. Using a perturbation argument, we show that any polymer that matches a certain number of moments with the log-gamma polymer also has Tracy-Widom fluctuations in intermediate disorder. 
\end{abstract} 

\tableofcontents

\begin{oldnotes}
Nov 26 2017 
\begin{enumerate}
    \item \sout{I'm looking at \verb|arxiv_and_aap/paper_v1krishnanquastel (1).pdf|. I think I've implemented all of these changes. I missed a few and I fixed those.}
    \item \sout{Should add somewhere that we identify the correct scale of fluctuations in intermediate disorder.}
    \item I looked at \verb|paper_v1krishnanquastel (2).pdf| and I'm looking at \verb|paper_v1krishnanquastel (3).pdf| right now. 
    \item \sout{I'm still trying to fix that argument on page 16, the proof of lemma 4.2}
\end{enumerate}
\end{oldnotes}

\section{Introduction}

In 1999-2000 \citet{MR1682248} and \citet{johansson_shape_2000} proved that the asymptotic fluctuations of the maximal energy (passage-time) in certain point-to-point last-passage problems were governed by the same Tracy-Widom law which arises in the large
$N$ limit of the top eigenvalue of an $N\times N$ matrix from the Gaussian Unitary Ensemble (GUE). It was then conjectured that this holds for very general distributions, and furthermore that it extends to the asymptotic free energy fluctuations of directed polymers in $1+1$ dimensions; i.e., the \emph{positive temperature} case. Here, the free energy takes the form 
\begin{equation}\label{1}
    F(\beta,N) = \log   \sum_{\path} \exp\left({\beta \sum_{i=1}^N  \wt_{\path(i)}}\right),
\end{equation}
where the up-right lattice paths $\path$ go from $(1,1)$ to $(N,N) \in \Z^2$, $\beta > 0$ is the inverse temperature, and the $\{\xi_{i,j}\}_{(i,j) \in \Z^2}$ are independent identically distributed (iid) random variables, collectively referred to as the disorder. 

To date, the only progress that has been made on the positive temperature conjecture is:  1) It has been verified  for the special, exactly-solvable \loggamma{} case \cite{MR2917766} in a certain low temperature range ($\beta > \beta^*$) \cite{MR3116323}, and 2) It has been shown to hold under certain double scaling regimes a) for  long thin rectangles
\cite{auffinger_universality_2012}, and b) in a special case of the intermediate disorder limit \cite{MR3189070}.   

In $1+1$ dimensions, the directed random polymers are in the strong disorder regime for all values of inverse temperature $\beta>0$ \citep{MR2271480}. The intermediate disorder regime means to take $\beta\to 0$ with the length of the polymer to probe the transition:  The more slowly $\beta$ is taken to $0$, the closer one is to the Tracy-Widom GUE asymptotics at fixed $\beta > 0$. The special case where $\beta_N=\hat\beta N^{-1/4}$ was studied in detail in \cite{MR3189070}. This is a double scaling regime involving the Kardar-Parisi-Zhang (KPZ) equation, where the fluctuations crossover from the Gaussian (Edwards-Wilkinson) regime to the Tracy-Widom GUE regime as $\hat\beta \to \infty$. In this article we use a standard perturbation argument (Theorem \ref{thm:perturbation theorem}) which shows the universality of the Tracy-Widom GUE distribution when $1\gg\beta_N\gg\mathcal{O}(N^{-1/4})$. 

If we fix some sequence $\beta_N$ in the last regime, the perturbation theorem says that if two disorder distributions have moments that are sufficiently close up to a certain explicitly identified order, then they have the same asymptotic free energy fluctuations. In principle, one would like to use this to prove the universality of the Tracy-Widom law in intermediate disorder for directed polymers free energies of the form \eqref{1}.  However, the only case in which the Tracy-Widom law is known, the \loggamma{} polymer, is not really of the form \eqref{1}. 

A \loggamma{} random variable is the log of a gamma distributed random variable; i.e., it has the \expgamma{} distribution. The \expgamma{} distributions form a two parameter family corresponding to the scale and shape parameters of the gamma distribution. The scale is a trivial parameter in the directed polymer since it corresponds to centering the weights. The shape parameter $(\theta)$ affects the properties of the \expgamma{}  distribution nonlinearly. However, at least at high temperature $(\theta \to \infty)$, the shape parameter approximately controls the centered moments just like the inverse temperature $\beta$ does in the standard polymer (see \eqref{thetatobeta}). Since the shape/temperature parameter of the \loggamma{} random variable does not appear multiplicatively, the statement (Corollary \ref{cor:perturbation theorem for log-gamma}) is not as simple as it would be if there were a solvable model
of the form \eqref{1}. Nevertheless, the result shows that the \loggamma{} polymer can be significantly perturbed in the intermediate disorder regime without changing the GUE Tracy-Widom fluctuations (see \Remref{rem:loggamma satisfies the moment bound}).

Finally, it turns out that the free energy fluctuations of the \loggamma{} polymer in intermediate disorder is outside of the range of the best available result \cite{MR3116323}, which requires $\beta\ge \beta^*>0$.  Most of the present article is devoted to removing this restriction, caused by the form of the contours employed in the exact formula for the \loggamma{} polymer in \cite{MR3116323}. We start with a different exact formula from \citep{MR3366125} that has more convenient contours, and we thank I.\,Corwin for pointing us towards this paper.  We also thank an anonymous reviewer and I.\,Corwin for comments about a small error in the Theorem from \citep{MR3366125}. Since we rely on this theorem, we sketch a way to fix their oversight in \Remref{rem:the error in the bcfv contour}.
 
 In this way, we obtain the Tracy-Widom GUE law for the point-to-point \loggamma{} polymer for all temperature parameter values and appropriate ``nearby'' distributions.

\section{Main Results}
\label{sec:statement-of-results}
We now describe precisely the discrete random polymer model. The disorder is a random field given by variables $\wt_{i,j}(\beta)$, $i,j\in \{1,2,\ldots\}$ which are independent for each $\beta>0$.  The polymer is represented as an  up-right directed lattice path $\path$ from $(1,1)$ to $(N,N)$. The energy of such a path is given by
\[
H_\beta(\path) = -\sum_{(i,j)\in \path} \wt_{i,j}(\beta).
\]
The partition function is given by
\begin{equation}
    Z_\beta^{N} = \sum_\path e^{-H_\beta(\path)}.
    \label{eq:polymer-partition-function-definition}
\end{equation}
Typically one would have $\wt_{i,j}(\beta)=\beta\wt_{i,j}$, but since we want to also consider the \loggamma{} polymer, we allow for a more complicated dependence on $\beta$.
The limiting free energy is given by 
\[
    F(\beta) = \lim_{N\to \infty} \frac1{N} \log Z^{N}_\beta.
\]
\citet[Proposition 1.4]{MR1939654} proved that the limit exists when the $\wt_{i,j}$ are iid Gaussians. \citet[Proposition 2.5]{MR1996276} extended their result to general iid weights with exponential moments.

The scaled and centered free energy fluctuations are given by
\begin{equation}
  h_N(\beta) := \frac{\log \partfunc - NF(\beta) }{\sigma(\beta) N^{1/3}},
  \label{eq:scaled-log-partition-functions}
\end{equation}
where in general, one expects the right scaling for $\sigma(\beta)$ to be
\begin{equation}
    \sigma(\beta) \approx C\beta^{4/3} \qquad{\rm as}\qquad \beta\searrow 0,
    \label{eq:sigma behavior for small beta or large theta}
\end{equation}
with a constant $C$ depending only on the distribution of the weights $\wt$. This scaling was conjectured in \citep{alberts_intermediate_2010}, and we prove it in this paper for the \loggamma{} and nearby polymers.

We will primarily be interested in the
intermediate disorder regime, in which $\beta$ goes to zero as $N\to \infty$, but $\lim_{N\to\infty} \sigma(\beta)N^{1/3} >0$.  In particular, if
\begin{equation}\label{lbonbeta}
\lim_{N\to\infty} \sigma(\beta)N^{1/3} =\infty,
\end{equation}
we expect the fluctuations to be Tracy-Widom GUE. If $\lim_{N\to\infty} \sigma(\beta)N^{1/3}=0$  the fluctuations are Gaussian, as can be seen by doing a chaos expansion in $\beta$ and checking that only the leading term, linear in the noise, survives. If $\lim_{N\to\infty} \sigma(\beta)N^{1/3} = \hat\beta \in(0,\infty)$, the partition function converges to the solution of the stochastic heat equation. 

In case \eqref{lbonbeta} the limiting fluctuations are supposed to have the Tracy-Widom GUE law in wide generality, but the only case where there are any results is the special \loggamma{} polymer.  Here $e^{-\wt(\beta)} $ have the gamma distribution
(or $-\wt(\beta)$ have the \expgamma{} distribution), which is supported on $x>0$ with density
\begin{equation*}
P(e^{-\wt(\beta)} \in dx) = \frac1{\Gamma(\theta)} x^{\theta -1} e^{-x} dx,
\end{equation*}
where
\begin{equation}\label{thetatobeta}
    \theta = \beta^{-2}.
\end{equation}
We show in \Secref{sec: proof of perturbation theorem by lindeberg} that the $k$\textsuperscript{th} centered moment of the \expgamma{} distribution satisfies  $\E[(\xi(\beta) - \E[\xi(\beta)])^k] = \Theta(\beta^k)$ for all $k \geq 1$ as $\beta \to 0$. Here, $\Theta(\beta^k)$ means that there exist constants $c_1, c_2$ such that the quantity in question is bounded above and below by $c_1 \beta^k$ and $c_2 \beta^k$ respectively for all small enough $\beta$. This mimics the way in which the inverse temperature $\beta$ would enter the standard polymer $\wt(\beta) = \beta\wt$. In other words, choosing \emph{$\beta = \theta^{-1/2}$ in \eqref{thetatobeta} ensures that at high-temperature, $\beta$ plays the role of inverse temperature in the \loggamma{} model.} 

For the \loggamma{} polymer,
\begin{equation}\label{13}
F(\beta) = -2\Psi(\theta/2),\qquad \sigma(\beta) = (-\Psi''(\theta/2))^{1/3},
\end{equation}
where 
\begin{equation}\label{digamma}
    \Psi(\theta) = \frac{\Gamma'(\theta)}{\Gamma(\theta)} 
\end{equation}
is the digamma function. The limiting free-energy was identified in \citep{MR2917766} and the variance in \citep{MR3116323}. 

Our first theorem concerns the fluctuations of the \loggamma{} model.
\begin{theorem}  Let $-\wt_{i,j}(\beta)$ have the \expgamma{} distribution and $\beta_N\to \beta\in [0,\infty)$
such that $\sigma(\beta_N)N^{1/3}\to \infty$.
Then
\begin{equation*}
\lim_{N\nearrow\infty} \Prob(h_N(\beta_N)<r) = F_{\rm GUE}(r)
\end{equation*}
where $h_N$ is the scaled-centered log partition function in~\eqref{eq:scaled-log-partition-functions}, $\Prob$ is the probability of the disorder, and $F_{\rm GUE}$ is the GUE Tracy-Widom distribution.
\label{thm:main-gue-theorem}
\end{theorem}

\citet{MR3116323} proved \Thmref{thm:main-gue-theorem} for $\beta\ge \beta^*>0$, where $\beta^*$ is some unidentified but finite number.  Our result removes this restriction, and further extends it into the intermediate disorder regime.

Our next result extends the $\beta_N\searrow 0$ part of this result to ``nearby'' distributions.

\begin{define}[Moment matching condition]
    Two parametrized families of weights $\wt=\wt(\beta)$ and $\tilde\wt= \tilde\wt(\beta)$ are said to \emph{match moments up to order }$k$ if for some $C<\infty $ and for all sufficiently small $\beta$,
    \begin{equation*}
        | \E[  \wt^n ] - \E[ \tilde\wt^n ] | \leq C\beta^{k} \qquad n=1,\ldots,k-1,
    \end{equation*}
   and 
    \begin{equation}\label{kbd}
        | \E[ \wt^k ] |,   | \E[ \tilde\wt^k ]|  \leq C\beta^{k}.
    \end{equation}
    \label{def:moment-matching-condition}
\end{define}

Let $C^k(\R)$ be the space of functions on $\R$ whose derivatives up to order $k$ are uniformly bounded on all of $\R$.

\begin{lemma}
    Suppose two families of weights $\xi(\beta)$ and $\tilde\xi(\beta)$ match moments up to order $k$ (as in Definition \ref{def:moment-matching-condition}) and let $\varphi\in C^k(\R)$. Let $h_N(\beta)$ and $\tilde{h}_N(\beta)$ be the scaled-centered partition functions corresponding to $\xi$ and $\tilde{\xi}$. Then there is a $C<\infty$ such that,
    \begin{equation}\label{17}
        |\E[\varphi(h_N)] - \E[\varphi(\tilde h_N)] | \leq C\frac{N^{2}{\beta^{k}}}{\sigma(\beta)N^{1/3}}.
    \end{equation}
    \label{lem: comparison of partition functions}
\end{lemma}

Lemma~\ref{lem: comparison of partition functions} is proved in~\Secref{sec: proof of perturbation theorem by lindeberg}, and the following theorem is a consequence of \Lemref{lem: comparison of partition functions} and the fact that weak convergence is equivalent to convergence of expectations of all $C^k(\R)$ functions \citep[Theorem 2.1]{MR1700749}.

\begin{theorem}[Perturbation theorem] Suppose $\wt(\beta)$ and $\tilde\wt(\beta)$ match moments up to order $k$ (as in Definition \ref{def:moment-matching-condition})  and $\beta_N\searrow 0$ with
     \begin{equation}
       \lim_{N\nearrow\infty}  \frac{N^{2}\beta_N^{k}}{\sigma(\beta_N)N^{1/3} } =0.
    \label{eq:condition that determines the number of moments we need to match}
    \end{equation}
    Then
  \begin{equation*}
        \lim_{N \to \infty} \Prob( h_N(\beta_N) \leq r)  =  \lim_{N \to \infty} \Prob( \tilde h_N(\beta_N) \leq r).
    \end{equation*}
    \label{thm:perturbation theorem}
\end{theorem}

We do not expect this perturbation technique to extend to positive temperature. The reason it works here is because the $k$\textsuperscript{th} term of the Taylor expansion of the log-partition function is of order $\beta_N^k$ (see \Eqref{kbd} and \Eqref{eq:lindeberg-replacement-taylor-expansion}), and $\beta_N \to 0$ in intermediate disorder. 

\begin{cor}
Suppose $\wt(\beta)$ and $\tilde\wt(\beta)$ match moments up to order $k$ with $-\tilde\wt(\beta)$ having centered \expgamma{} distribution as above and 
\eqref{eq:condition that determines the number of moments we need to match} holds. 
    Then
  \begin{equation*}
        \lim_{N \to \infty} \Prob(  h_N(\beta_N) \leq r)  = F_{\rm GUE}(r).
    \end{equation*}
    \label{cor:perturbation theorem for log-gamma}
\end{cor}

 For example, if $\beta_N= N^{-\alpha}$, $\alpha\in (0,1/4]$, it is sufficient if $\xi(\beta)$ matches moments up to order  
 \begin{equation*}
        k > \frac{5}{3\alpha} + \frac{4}{3}
    \end{equation*}
  with $\tilde{\xi}$. 
  
  \begin{remark}
      It is not obvious that the \expgamma{} distribution satisfies the moment bound in \eqref{kbd}. This is addressed in the proof of \Corref{cor:perturbation theorem for log-gamma}. Given this fact, let $\{X_{i,j}\}$ be a family of independent random variables with $k$ bounded moments.
    Let  $\{ \tilde\xi_{i,j} \}$ be an iid family of centered \expgamma{} random variables with parameter $\theta = \beta^{-2}$ that are independent of the $\{X_{ij}\}$. Then the polymer with weights 
    \begin{equation*}
        \xi_{i,j}  = \tilde\xi_{i,j} (1 + X_{i,j} \beta^{k})
    \end{equation*}
    satisfies the moment matching condition, and hence its log-partition function has asymptotic GUE Tracy-Widom fluctuations when \eqref{17} holds.
      \label{rem:loggamma satisfies the moment bound}
  \end{remark}
  \begin{remark}
      In the standard polymer, it is known when $\alpha=1/4$ that $6$ moments suffice to get the crossover law goverened by the KPZ equation, so the result is slightly suboptimal. Note that our result requires an increasing number of moments to match as $\alpha \searrow 0$, whereas $5$ moments are thought to suffice when $\alpha =0$ \cite{MR2335697}. 
  \end{remark}

\begin{remark}
    The perturbation theorem can clearly be stated in higher dimensions; we only need to modify \Lemref{lem: comparison of partition functions}. However, the critical temperature and hence the intermediate disorder regime is positive at higher temperatures, and our perturbation argument would have to be modified. Therefore we do not pursue these issues here.
\end{remark}


    \Thmref{thm:main-gue-theorem} is proved in \Secref{sec:proof of main tracy widom theorem for loggamma polymer}. \Lemref{lem: comparison of partition functions} and \Corref{cor:perturbation theorem for log-gamma} are both proved in \Secref{sec: proof of perturbation theorem by lindeberg}. The appendix contains the proof of the Fredholm determinant formula we use in \Secref{sec:proof of main tracy widom theorem for loggamma polymer}, and fixes a small error in \citet[Theorem 2.1]{MR3366125}.
\section{Proof of the perturbation theorem}
\label{sec: proof of perturbation theorem by lindeberg}

The Lindeberg proof of the central limit theorem is now a standard argument for proving universality~\citep{MR1544569,Lindeberg1920}. Let $f$ be a function on $\R^n$ and consider two sets of iid random variables $\xi = (\xi_1,\ldots,\xi_n)$ and $\tilde{\xi} = (\tilde{\xi}_1,\ldots,\tilde{\xi}_n)$ that match some number of moments. For any bounded smooth function $\varphi$, the Lindeberg strategy allows one to show $|\E[\varphi \circ f(\xi) ] - \E[\varphi \circ f(\tilde{\xi}) ]| = o(n)$ for all smooth and bounded functions $\varphi$. This is shown by replacing the $\xi$ variables one by one with $\tilde{\xi}$ variables, and using Taylor expansion to control the error. This estimate controls the weak-$*$ distance between $f(\tilde{\xi})$ and $f(\xi)$ and thus shows that they converge to the same distributional limit if it exists for either one of them. The technique has been applied to show, for example, that the limiting free energy of the Sherrington-Kirkpatrick (SK) spin glass, and the semi-circle distribution in Wigner random matrices are \emph{universal}; i.e., they are independent of the distributions of the variables involved \citep{chatterjee_simple_2005}.

There is another related technique in spin glass theory called Guerra's interpolation method that also relies on Taylor expansion. It uses the Ornstein-Uhlenbeck process to interpolate between a vector of iid random variables $\tilde{\xi}$ and an independent iid \emph{Gaussian} vector $\xi$. In the SK model, the partition function is of the form $Z = \sum_{\sigma \in \{-1,1\}^N} \exp(-\beta_N H(\sigma))$ where $H(\sigma) = - \sum_{ij} \xi_{ij} \sigma_i \sigma_j$ and $\beta_N = N^{-1/2}$. Again, $N^{-1} \log Z$ has a deterministic limit called the free-energy as $N \to \infty$. For iid Gaussian weights, the limit was shown to be given by the celebrated Parisi formula by \citet{MR2195134}. The limit was shown to be the same for all iid families of symmetric random variables with four moments by \citet{MR1930572}. They used the aforementioned interpolation technique and the so-called approximate integration by parts for weights that match the moments of a Gaussian up to some order. Their ideas were extended by \citet{MR1939654} to include distributions that match two moments with the Gaussian and have finite third moment. Chatterjee used  truncation and the Lindeberg technique to remove the finite third moment requirement \citep{chatterjee_simple_2005}. Other results like \citep{auffinger_universality_2014} and \citep{MR3058201} extend the interpolation technique using higher order Taylor expansions. 

In particular, \citet{auffinger_universality_2014} showed that the \emph{Gibbs measure} ---the random measure on configurations given by $P_{\xi}(\sigma) = Z^{-1} \exp( \beta_N H(\sigma) )$--- also converged to a universal limit as long as the weights matched a certain number of moments with the Gaussian. Since their results are for general spin-systems and not just for the SK model, they also apply to polymer models in intermediate disorder.  In a personal communication, \cite{auffinger_universality_2014_polymers} applied the theorems in \citep{auffinger_universality_2014} to show that the limiting Gibbs measure associated with the polymer path is universal: 
Let $\beta_N = \beta N^{-\alpha}$, and let $\gamma = (\gamma_i)_{i=1}^{Nd}$ be a directed path from the origin to $N(1,\ldots,1) \in  \Z^d$, where $\gamma_i$ are the vertices along path. Suppose the weights $\tilde{\xi}$ in the polymer match the first $k$ moments of the standard Gaussian such that
\[
    \E[ \tilde{\xi}^{k+1} ] < \infty \quad k > \frac{d+1}{\alpha}.
\]
For $n \in \mathbb{N}$, let $\gamma^a$, $a=1,\ldots,n$ be any $n$ directed paths from the origin to $N(1,\ldots,1)$. 
\begin{theorem}[Auffinger \citep{auffinger_universality_2014_polymers}]
    Let $L$ be a function depending on $n$ paths $(\gamma^a)_{1\leq a \leq n}$ where $n$ is fixed, and suppose $\Norm{L}{\infty} \leq 1$. 
    Then,
    \[
        |\E_{\tilde{\xi}}\left[ \langle L(\gamma) \rangle \right] - \E_{\xi}\left[ \langle L(\gamma) \rangle \right]| \to 0 \quad \text{as } N \to \infty.
    \]
    Here, $ \langle \cdot \rangle$ represents the average over the Gibbs measure on paths, and $\E_{\cdot}$ represents expectation over the corresponding set of weights. 
\end{theorem}
This allowed him to show that various quantities of interest were universal, including the transversal fluctuation exponent of the path measure that is defined as follows: Let $\gamma_{N/2}$ be the midpoint of a path $\gamma$ sampled from the Gibbs measure. The polymer has transversal fluctuation exponent $\alpha$ if for any $\alpha' < \alpha < \alpha''$ and $C > 0$,
\begin{multline}
    \Prob_{\tilde{\xi}} \left( |\gamma_{N/2} - N/2(1,\ldots,1)| \leq C N^{\alpha'} \right) \to 0 , \\ \AND  
    \Prob_{\tilde{\xi}} \left( |\gamma_{N/2} -N/2(1,\ldots,1)|  \leq C N^{\alpha''} \right) \to 1 \quad \text{as} \quad N \to \infty.
\end{multline}
The transversal exponent has a relationship with $\chi$, the fluctuation exponent of $\Var \left( \log Z^N(\beta) \right) \sim N^{2 \chi}$. Under certain strong hypotheses on the \emph{existence} of these exponents (they are not known to exist for any standard polymer) it is known that $\chi = 2\alpha - 1$ \citep{chatterjee_universal_2011,MR3141825}. Therefore, \citep{auffinger_universality_2014_polymers} also indicates that polymers ought to have the same fluctuation exponents as the Gaussian polymer under a moment matching condition. We prove that $\chi$ is universal for the \loggamma{} and nearby polymers, verifying a conjecture of \citep{alberts_intermediate_2010} in this special case.

In our setting, there are a few problems with Guerra's interpolation technique. \emph{In its current form, it only allows one to match moments with the Gaussian}. Hence one can only look at polymers where the weights are of the form $\xi_{i,j}(\beta) = \beta \xi_{i,j}$ and only match moments with the Gaussian distribution. This by itself is not a very serious shortcoming and may be overcome with some work. 

But for the Gaussian polymer, it is not known whether the fluctuation exponents discussed above exist, and whether the fluctuations of $\log Z^N(\beta)$ are in the GUE Tracy-Widom universality class. In fact, very little is known about the Gaussian polymer other than the fact that the limiting free energy exists. 

On the other hand, a lot more is known about the log-gamma polymer. The scale of the variance \citep{MR2917766}, and the limiting fluctuations are known in a large parameter range \citep{MR3116323}. Moreover, the Lindeberg replacement technique is well-suited to comparing other polymers with the \loggamma{} polymer. Since we need more terms in the Taylor expansion (than \citep{chatterjee_simple_2005}, for example), and since Guerra's Gaussian integration by parts technique does not apply directly to prove~\Lemref{lem: comparison of partition functions}, we reproduce the fairly standard Lindeberg argument for completeness. 

\begin{proof}[Proof of \Lemref{lem: comparison of partition functions}]
  For a fixed vertex $x = (i,j)$, define
    \begin{equation}\label{zwhy}
    Z(y) = Z_{x^c} + Z_x e^y,
    \end{equation}
    where $Z_{x^c}$ represents the sum in~\Eqref{eq:polymer-partition-function-definition} over paths that do not pass through $x$, and $Z_x$ is the sum over paths that do pass through $x$, but do not include weight at $x$. Let $h(y) = N^{-1/3}\sigma^{-1} (\log Z(y) - NF)$.  $Z(y)$ and $h(y)$ do indeed depend on the other weights $\xi_z$ for $z \neq x$, but the dependence is suppressed in the notation because we want to isolate the effect of replacing $\wt_x$ by $\tilde\wt_x$. For any function $\varphi \in C^k$ we will show that
    \begin{equation}
        |\E[\varphi(h(\wt_x))] - \E[ \varphi(h(\tilde\wt_x)) ] | \leq C(\sigma N^{1/3})^{-1}\beta^k
        \label{eq:single-lindeberg-replacement-in-polymer}
    \end{equation}
    where the expectation is over the disorder. We obtain \eqref{17} by replacing $\wt_{x}$ by $\tilde\wt_x$, $N^2$ times for each $x \in \{1,\ldots,N\}^2$.

    Fix all the other weights in the disorder, and write Taylor's theorem for $\varphi(h(\wt_x))$:
    \begin{equation*}
      \varphi(h(\wt_x)) = \sum_{j=0}^{k-1} \frac{\partial_y^j\varphi(h(0))}{j!}  \wt_x^j +  \frac{\partial_y^k\varphi(h(\zeta))}{k!} \wt_x^k,
    \end{equation*}
    with $\zeta$ between $0$ and  $\wt_x$. Taking expectation and using the independence of $\{\wt_{z}\}_{z \in \R^2}$, we get
    \begin{equation}
        \E[  \varphi(h(\wt_x)) ] = \sum_{j=0}^{k-1} \frac{a_j}{j!}  \E[\wt_x^j] + \frac{a_k}{k!} \E[ \wt_x^k],
        \label{eq:lindeberg-replacement-taylor-expansion}
    \end{equation}
    where $a_j = \E[\partial_y^j\varphi(h(0))]$, $j=1,\ldots, k-1$ and $a_k=\E[\partial_y^k\varphi(h(\zeta))]$.
    One has an analogous expression for $ \E[  \varphi(h(\tilde\wt_x)) ]$, but note that in fact $a_j=\tilde a_j$ for
    $j=1,\ldots, k-1$ since they both do not depend on $\xi_x$ and $\tilde\xi_x$, and all the other weights they depend on are the same. Hence, from the moment matching condition \eqref{kbd},
   \begin{equation*}
        | \E[  \varphi(h(\wt_x)) ] -  \E[  \varphi(h(\tilde\wt_x)) ]|\le  \left(\sum_{j < k} |a_j| + (|a_k|+ |\tilde a_k|) \right) C \beta^k.
    \end{equation*}
       To control the error term, we will show that for any $k \geq 1$ and all $y\in \R$,
    \begin{equation}
        |\partial_y^k\varphi(h(y))| \leq C_{k,\varphi} (\sigma N^{1/3})^{-1},
        \label{eq:lindeberg-expansion-derivative-bound-for-composition}
    \end{equation}
    where $C_{k,\varphi}$ is a constant dependent only on $\varphi$, $k$ and the constant from the moment matching condition. The estimates \eqref{eq:lindeberg-replacement-taylor-expansion}, \eqref{eq:lindeberg-expansion-derivative-bound-for-composition},  and the moment matching condition in~\Defref{def:moment-matching-condition} together imply \eqref{eq:single-lindeberg-replacement-in-polymer}.

To prove \eqref{eq:lindeberg-expansion-derivative-bound-for-composition}, we expand the derivative of a composition (\`a la Faa di Bruno)
\begin{equation*}
\partial^k\varphi(h) = \sum_{\sum s m_s = k} C_{m_1\ldots m_k} \partial^{\sum m_s}\varphi \prod_{r=1}^k (\partial^r h)^{m_r},
\end{equation*}
where the $C_{m_1\ldots m_k}$ are multinomial coefficients, and $m_s \geq 0$ for $s=1,\ldots,k$. Since $\varphi$ is smooth with bounded derivatives up to order $k$, we only need to control $\partial^rh(0)$ for $r\ge 1$.
Computing derivatives in  \eqref{zwhy},
\begin{equation*}
    \begin{split}
        \partial_y \log Z(y) & = \frac{Z_x e^y }{Z_{x^c} + Z_x e^y} =: p(y),\\
        \partial_y^i \log Z(y) & = \mathcal{P}_i(p(y)), \quad i > 1,
    \end{split}
\end{equation*}
where $\mathcal{P}_i$ is the polynomial given by the recurrence
\begin{equation*}
    \mathcal{P}_{i+1}(p) = \mathcal{P}_{i}'(p)\,p(1-p), \quad \mathcal{P}_1(p) = p, \quad i \geq 1.
\end{equation*}
The recurrence follows from the chain rule and $p'(y) = p(y)(1-p(y))$. Since $0 \leq p(y) \leq 1$ for all $y \in \R$, we can bound each of the polynomials $\mathcal{P}_i$ by constants for $i=1,\ldots,k$. Putting the last few observations together, we get~\Eqref{eq:lindeberg-expansion-derivative-bound-for-composition} for $k \geq 1$.
\end{proof}

\begin{proof}[Proof of \Corref{cor:perturbation theorem for log-gamma}]
    \Thmref{thm:main-gue-theorem} shows that if the $-\tilde{\wt}(\beta)$ are \expgamma{} random variables, and $\beta_N \to \infty$ such that \eqref{eq:condition that determines the number of moments we need to match} holds, then
    \[
        \lim_{N \to \infty} \Prob(  h_N \leq r)  = F_{\rm GUE}(r).
    \]
    We only need to show that $-\tilde{\wt}(\beta)$ satisfies the moment bound in \eqref{kbd} (see \Remref{rem:loggamma satisfies the moment bound}), which says that for $\theta = \beta^{-2}$
    \begin{equation}
        |\E[ (\tilde{\wt}(\beta) - \E\tilde{\wt}(\beta))^k ]| \leq C_k \frac1{\theta^{\lceil k/2 \rceil}},
        \label{eq:bound for central moments of log gamma}
    \end{equation}
    for some constant $C_k$. We will in fact show that for all $k > 1$,
    \begin{equation}
        \E[ (\tilde{\wt}(\beta) - \E\tilde{\wt}(\beta))^k ] = \Theta\left( \frac1{\lceil \theta^{ k/2 } \rceil} \right), \quad \theta \to \infty.
        \label{eq:bound for central moments of log gamma identifing correct order}
    \end{equation}
    
    The cumulant generating function of the \expgamma{} distribution is given by
\begin{equation*}
    \log \E[ \exp( t X ) ] = \log\left( \frac{\Gamma(t + \theta)}{\Gamma(\theta)} \right),
\end{equation*}
where $X$ is a gamma distributed random variable.
Differentiating this $k$ times with respect to $t$, we see that the $k$\textsuperscript{th} cumulant of the \expgamma{} distribution is given by $\Psi^{(k-1)}(\theta)$, the $k-1$\textsuperscript{th} derivative of the digamma function \Eqref{digamma}.  The digamma function can be written as \citep[6.3.16]{MR0167642}
\begin{equation}
    \label{eq:series representation for digamma function}\Psi(z)  = - \gamma_{EM} + \sum_{n=0}^{\infty} \left(\frac{1}{n+1} - \frac{1}{n+z}\right),
\end{equation}
where $\gamma_{EM}$ is the Euler-Mascheroni constant.  Note that the series is absolutely convergent when $z$ is bounded away from the nonpositive integers. It follows that 
    \begin{equation*}
        \kappa_k := \Psi^{(k-1)}(\theta) = \Theta \left( \frac1{\theta^{k-1}} \right) \quad k > 1, \quad \theta \to \infty.
    \end{equation*}
    For any random variable, the moments $\mu_n$ are related to the cumulants $\kappa_n$ via the following combinatorial expansion (see \citep{MR725217} for this formula and its famous M\"obius inversion). 
    If $\pi$ is a (set) partition of $\{1,\ldots,k\}$, then we represent $\pi$ as the union of disjoint sets $\pi = \{B_i\}_{i=1}^{n(\pi)}$ where $B_i \subset \{1,\ldots,k\}$, and $\cup B_i = \{1,\ldots,k\}$. Then, 
    \[
        \mu_k = \sum_{\pi \in \mathcal{L}} \prod_{B \in \pi} \kappa_{|B|}  
    \]
    where $|B|$ represents the cardinality of the set $B$, and $\mathcal{L}$ is the set of all partitions of $\{1,\ldots,k\}$. Since we are considering centered random variables, we simply set the first cumulant to zero ($\kappa_1 = 0$) and therefore get
    \begin{align}
        |\mu_{k} |
        & = \left| \sum_{\pi \in \mathcal{L}} \prod_{i=1}^{n(\pi)} \kappa_{|B_i|} 1_{|B_i| \neq 1} \right| \nonumber \\ 
        & \leq \sum_{\pi \in \mathcal{L}} \prod_{i=1}^{n(\pi)} \frac{C_{|B_i|}}{\theta^{|B_i|-1}} 1_{|B_i| \neq 1}, \nonumber \\
        & = \sum_{\pi \in \mathcal{L}} \frac{C'_{\pi}}{\theta^{\sum_{i=1}^{n(\pi)}|B_i| - n(\pi)}} 1_{|B_i| \neq 1}, \nonumber  \\
        & \leq \frac{C_k''}{\theta^{k - \max_{\pi, |B_i| \neq 1} n(\pi)}} \label{eq:where we take the max over a set of partitions in the cumulant calculation} \\
        & = \frac{C_k''}{\theta^{\lceil k/2 \rceil}}, \nonumber
    \end{align}
    where $C'_{\pi} = \prod_{i=1}^{n(\pi)} C_{|B_i|}$ and $C''_k = 2^k \max_\pi C'_{\pi}$. In \eqref{eq:where we take the max over a set of partitions in the cumulant calculation}, we used the following observation: when $\pi$ contains no sets with only one element, it follows that $n(\pi) \leq \lfloor k/2 \rfloor$. This proves \eqref{eq:bound for central moments of log gamma}.
    
    When $k$ is even, the leading order term for $\mu_k$ comes from partitions whose $B_i$ have exactly $2$ elements for all $i$. If there is even one $B_i$ with more than $2$ elements, then $n(\pi) \leq k - 2$. Thus, all the other partitions result in terms that have strictly smaller order as $\theta \to \infty$. Thus, we get $\mu_k \geq c_k\theta^{ -k/2 }$. A similar argument for $k$ odd applies; here partitions that have one $|B_i|=3$ and all the rest having two elements provide the leading order term. This proves \Eqref{eq:bound for central moments of log gamma identifing correct order}. 
\end{proof}

\section{Tracy-Widom fluctuations for the \loggamma{} polymer}
\label{sec:proof of main tracy widom theorem for loggamma polymer}

In this section, we prove \Thmref{thm:main-gue-theorem}. We begin with the Fredholm determinant formula for the Laplace transform of the partition function.

\subsection{Fredholm determinant formula}
\begin{theorem}\label{theorem4.1}
    For $N \geq 9$, let $\partfunc$ be the partition function of the \loggamma{} polymer with $\theta = \beta^{-2}$.  Then for $\Re(u)>0$, 
    \begin{equation}\label{32}
        \E[e^{-u \partfunc}] = \det(I + \Kloggamma)_{L^2(\Cv{\varphi} )},
    \end{equation}
    where
    \begin{equation}
    \begin{aligned}
        \Kloggamma(v,v')
        & = \frac{1}{2\pi \I} \int_{l_{\zcrit + \delta}} - \frac{\pi}{\sin(\pi (w-v))} \left( \frac{\Gamma(v)}{\Gamma(w)} \frac{\Gamma(\theta - w)}{\Gamma(\theta - v)} \right)^N \frac{u^{w-v} }{w -v'}~dw 
        + \sum_{i=1}^{q(v)} \Res_{u,i}(v,v'),
         \end{aligned}
    \label{eq:BCFV kernel with tau equals 0}
\end{equation}
\begin{equation}
    \Res_{u,j}(v,v') = (-1)^j \left( \frac{\Gamma(v)}{\Gamma(v+j)} \frac{\Gamma(\theta - v - j)}{\Gamma(\theta - v)} \right)^N \frac{u^{j} }{v + j -v'}, \quad 1 \leq j \leq q(v),
    \label{eq:general residues of the sine in our fredholm determinant formula}
\end{equation}
and
\begin{equation}
    q(v) = \lfloor \zcrit + \delta - \Re(v) \rfloor, \quad
    \zcrit ={\theta}/{2},  \quad 0 < \delta \leq \frac{\zcrit}{2}.
    \label{eq:number of residues and zcrit definition}
\end{equation}
The contours are defined as follows: For any $\varphi\in (0,\pi/4]$, the $\Cv{\varphi}$ contour is given by
$\{\zcrit+e^{\I (\pi+\varphi)}y\}_{y\in \R^+}\cup \{\zcrit+e^{\I(\pi-\varphi)}y\}_{y\in \R^+}$, where $\R^+$ is the set of non-negative reals. The $\Cw{x}$ contour is a vertical straight-line with real part $x$ (see Figure \ref{fig}). Both $\Cw{x}$ and $\Cv{\varphi}$ are oriented to have increasing imaginary parts. 
    \label{thm:our version of BCFV formula with tau equals 0}
    \oldnotes{I'd originally had $\varphi\in (\operatorname{arccot}(2),\pi/4]$ since I thought it appeared in~\Propref{prop:our estimate for the BCFV kernel for the Hadamard bound}. But I can't seem to track down where exactly it was needed, and so I'm going to drop it.}
\end{theorem}
\begin{oldnotes}
    The formula $\Gamma(v + 1) = v \Gamma(v)$ always works because it is used to analytically continue the usual Gamma function~\href{http://dlmf.nist.gov/5.5}{(the formula is here)}. I'd originally used the formula to simplify the residues in~\Eqref{eq:general residues of the sine in our fredholm determinant formula}, but this turned out to not be that useful.
\end{oldnotes}

\begin{figure}[!h]
    \def\svgwidth{3.2in}
    \hspace{1.4in}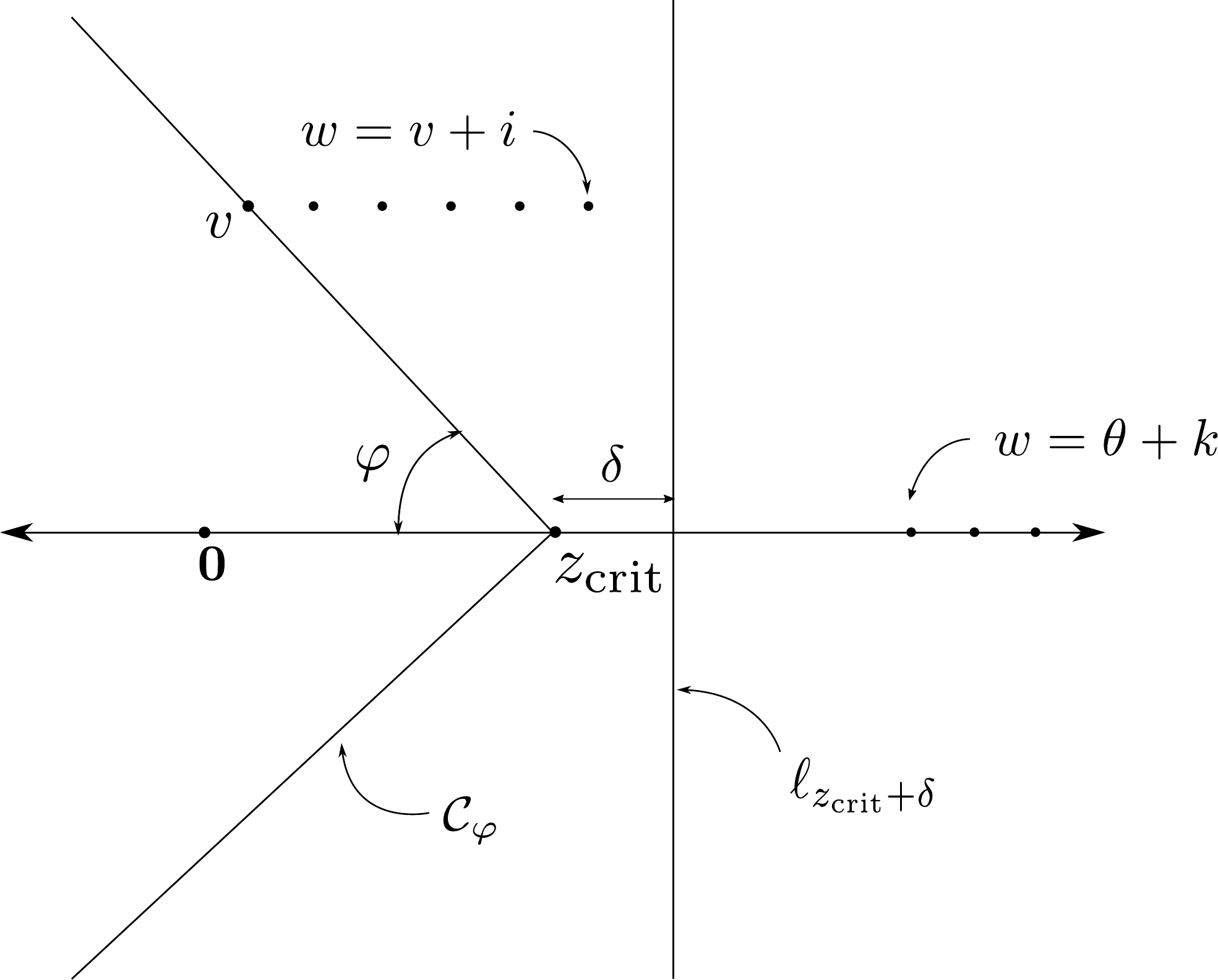
    \caption{Contours in~\Thmref{theorem4.1}. The critical point is at $\zcrit = \frac{\theta}{2}$. The triangular contour is $C_{\varphi}$, and the vertical contour $l_{\zcrit + \delta}$ has real part $\zcrit + \delta$. There are two sets of poles that one needs to watch out for. The poles of the sine are shown as dots to the right of $v$ and are of the form $w = v + i$. The poles of  $\Gamma(\theta - w)$ are of the form $w = \theta + k$, $k=0,1\ldots$.} 
    \label{fig}
\end{figure}

\begin{remark}
    Theorem~\ref{thm:our version of BCFV formula with tau equals 0} is proved by setting $\tau = 0$ in \citep[Theorem 2.1]{MR3366125}, as suggested in Remark 2.9 of the same paper. This requires a new estimate, and is done in \Secref{sec:appendix on fredholm determinant}. 
    
    We will see in the next section that the critical point of the integrand of $\Kloggamma$ in~\Eqref{eq:BCFV kernel with tau equals 0} is at $\zcrit$. The contours $\Cw{\zcrit + \delta}$ (as $\delta \to 0$) and $\Cv{\varphi}$ are located at the critical point. 
\end{remark}


\subsection{Estimates along the contours}
\label{sec:estimates along the contours}
We are interested in the asymptotic probability distribution of \eqref{eq:scaled-log-partition-functions}. The trick is to rewrite the left-hand side of \eqref{32} as 
 \begin{equation}\label{34}
  \E\left[\exp\left( -e^{\sigma N^{1/3} (h_N-t)} \right)\right]
 \end{equation}
 by taking
 \begin{equation}
  u =e^{- NF - t\sigma N^{1/3}}.
  \label{eq:u definition in intermediate disorder}
\end{equation}
As $N\nearrow\infty$, by \eqref{lbonbeta} $\sigma N^{1/3} \nearrow \infty$, and \eqref{34} becomes
$\lim_{N\nearrow \infty}\Prob( h_N \leq t)$.
Now we consider the same limit of the right-hand side of \eqref{32}. We start with a formal critical point analysis of the integral in \Eqref{eq:BCFV kernel with tau equals 0}, which can be rewritten as
\begin{equation}\label{BCRformula}
   -\frac1{2\pi\I}\int_{\Cw{\zcrit + \delta}}\frac\pi{\sin(\pi(w-v))}
e^{ N\left[G(v)-G(w)\right] + t \sigma N^{1/3}(v-w)}\frac{dw}{w-v'}
\end{equation}
where we have ignored the residues, dropped the subscript $u$ in the kernel, and let
\begin{equation}
  G(z)  = \log \Gamma(z) -\log \Gamma(\theta-z)  + F(\beta) z.
  \label{eq:u-G-f-in-intermediate-disorder}
\end{equation}
We have
$
G'(z) = \Psi(z) + \Psi(\theta-z) +F(\beta)$.
From~\eqref{13}, it follows that the critical point, i.e. $G'(\zcrit) =0$, is at $\zcrit=\theta/2$, 
 and $G''$ vanishes there as well. Therefore
the exponent is cubic near the critical point and it is natural to define
\begin{equation}
\tilde{v} = \sigma N^{1/3}(v-\zcrit), \qquad \tilde{w} = \sigma N^{1/3}(w-\zcrit),
\label{eq:rescaling the variables in the loggamma kernel to take a limit}
\end{equation}
and let $\tilde{K}^N(\tilde{v},\tilde{v}') = \Ksteepestdescent(v,v')$ in \Eqref{eq:BCFV kernel with tau equals 0}. The change of variable introduces a Jacobian factor of $(\sigma N^{1/3})^{-1}$ into the Fredholm expansion \Eqref{eq:fredholm series expansion}. Then, it is easy to prove the following lemma. 
\begin{lemma}
\begin{equation}
    \lim_{N \to \infty} \frac{1}{(\sigma N^{1/3})} \tilde{K}^N(\tilde{v},\tilde{v}') =  \Kairy(\tilde{v},\tilde{v}'),
    \label{eq:statement of pointwise convergence to airy kernel}
\end{equation}
where the Airy kernel is defined as
\begin{equation}\label{Krkappa}
    \Kairy(\tilde{v},\tilde{v}') := \frac1{2\pi\I}\int \frac{\exp\left\{-\tfrac{1}{3}\tilde{v}^3+ t \tilde{v}\right\}}
    {\exp\left\{\tfrac{1}{3}\tilde{w}^3+ t \tilde{w}\right\}}
    \frac{d\tilde{w}}{(\tilde{v}-\tilde{w})(\tilde{w}-\tilde{v}')}.
\end{equation}
    \label{lem:pointwise convergence of loggamma kernel to airy kernel}
The Airy kernel acts on the contour $\{e^{-2\pi\I/3}\R^+\cup e^{2\pi\I/3}\R^+\}$ and the integral in $\tilde{w}$
is on the contour $\{e^{-\pi\I/3}\R^+ + \delta\} \cup \{e^{\pi\I/3}\R^+ + \delta \}$ for any horizontal shift
$\delta>0$. Both are oriented to have increasing imaginary part.
\end{lemma}
The Fredholm determinant of the right hand side of \eqref{Krkappa} is $F_{\rm GUE}(t)$~\citep{MR3116323}. We prove \Lemref{lem:pointwise convergence of loggamma kernel to airy kernel} and flesh out the details of this sketch in the rest of this section.

Next, we upgrade the pointwise convergence in \Lemref{lem:pointwise convergence of loggamma kernel to airy kernel} to prove that $\det(\Id + \Ksteepestdescent) \to \det(\Id + \Kairy)$. Recall the kernel~\eqref{eq:BCFV kernel with tau equals 0} 
\begin{equation}
    \Ksteepestdescent(v,v') = \frac{1}{2\pi \I} \int_{\Cw{\zcrit + \delta(\sigma N)^{-1/3}}} I(v,v',w - v) dw + \sum_{i=1}^{q(v)} \Res_{i}(v,v'),
    \label{eq:log gamma kernel K for in position for asymptotic analysis}
\end{equation}
where $I(v,v',w-v)$ is the integrand in~\Eqref{BCRformula}. We drop the subscript $u$ on $\Ksteepestdescent$, $I$ and $\Res_i$ in this section to indicate that we have set $u$ as in~\Eqref{eq:u definition in intermediate disorder}. The kernel acts on the $\Cv{\varphi}$ contour as before, and we set $\varphi = \pi/4$. The little extra displacement of the $\Cw{\zcrit + \delta(\sigma N)^{-1/3}}$ is a necessary technicality that we will address in due course. Therefore we will henceforth simply write $\Cw{\zcrit}$ as a shorthand.

For $(v,v')$ on the $\Cv{\pi/4}$ contour, we show that
\begin{equation}
    |\Ksteepestdescent(v,v')|\leq f(v,N),
    \label{eq:exponential bound along contours for the kernel}
\end{equation}
where $f(v,N)$ is defined in \Lemref{lem:rigorous hadamard bound on steepest descent contours}. Then, from the Hadamard inequality for determinants, we get for $m > 1$,
\begin{equation}
    \left|\det(\Ksteepestdescent(v_i,v_j))_{1\leq i,j\leq m}\right|\leq \prod_{i=1}^m f(v_i,N) m^{m/2}.
    \label{eq:hadamard-inequality-to-bound-determinant-independent-of-N}
\end{equation}
$f(v,N)$ depends on $v$ and $N$ in such a way that it integrates over $\Cv{\pi/4}$ to a quantity that is bounded above by a constant independent of $N$.  
It follows that the Fredholm expansion of the determinant
\begin{align}
    \det(I+\Ksteepestdescent)_{L^2(\Cv{\pi/4})} 
    & =\sum_{m=0}^\infty \frac{1}{m!} \int_{\Cv{\pi/4}} dv_1 \cdots \int_{\Cv{\pi/4}} dv_m \det(\Ksteepestdescent(v_i,v_j))_{1\leq i,j\leq m},\label{eq:fredholm series expansion}
\end{align}
is absolutely summable uniformly in $N$ (see \Eqref{eq:rigorous bound for mth term in fredholm series}). Thus, we can take the $N\to\infty$ limit inside the series and integrals and replace $\Ksteepestdescent$ by its pointwise limit. This is similar to what was done in \citet{MR3116323}, but now the quantity in
\eqref{eq:hadamard-inequality-to-bound-determinant-independent-of-N} must be bounded uniformly in $\theta$ as well as $N$, since $\theta$ can go to infinity with $N$ (see \Thmref{thm:main-gue-theorem}). The rigorous estimates are shown in~\Lemref{lem:rigorous hadamard bound on steepest descent contours}. \Lemref{lem:pointwise convergence of loggamma kernel to airy kernel} and \eqref{eq:exponential bound along contours for the kernel} require several estimates on contours that appear in Lemmas \ref{lem: decay of G along the Cv contour}, \ref{lem:decay of G along the Cs contour} and \ref{lem:residues are small}. These estimates are summarized in the following 4 steps below.

Recall the function $G(z)$ defined in~\Eqref{eq:u-G-f-in-intermediate-disorder}. Using \Eqref{13}, we write is as
\begin{equation}
    G(z) = \log \Gamma(z) -\log \Gamma(\theta-z)  - 2 \Psi(\zcrit) z.
    \label{eq:G function second definition}
\end{equation}
The bound in~\Eqref{eq:exponential bound along contours for the kernel} will follow from an analysis of this function along the contours $\Cv{\pi/4}$ and $\Cw{\zcrit}$, and an estimate on the residues $\Res_{i}(v,v')$. The constants in the estimates are independent of $N$, but they do depend on $\theta$. As long as $\theta \geq \theta_0 > 0$, the constants are well-behaved; this is guaranteed by $\sigma(\beta) N^{1/3} \to \infty$ in \eqref{eq:condition that determines the number of moments we need to match}.

In the following, we set $\tilde\sigma = \sigma(\beta)N^{1/3}$, and note that
\begin{equation}
    \tilde\sigma = \Theta((N \zcrit^{-2})^{1/3}).
    \label{eq:order of magintude of tilde sigma}
\end{equation}
\begin{enumerate}
    \item \label{step:G taylor expansion in intermediate regime} In \Propref{prop:estimates for third and fourth derivative of G close to critical point}, we first show that the Taylor approximation is  effective  in the region $|z - \zcrit| \leq c \tilde\sigma^{-1}$. Although $G$ is analytic, one must be careful because the derivatives of $G$ are a function of $\theta$, which is allowed to go to infinity with $N$.

        Using the Taylor expansion, we may also arrange for an estimate of the form (recall $\zcrit = \theta/2$)
        \begin{equation*}
            \Re(G(z) - G(\zcrit)) \leq - \frac{C}{\zcrit^2} |z - \zcrit|^3, \quad |z - \zcrit| \leq \frac{\zcrit}{2},
        \end{equation*}
        where $C>0$  can be explicitly chosen. This is needed to show that the pointwise limit of $\Ksteepestdescent$ is $\Kairy$.

    \item \label{step:G is a decreasing along the contour} In \Lemref{lem: decay of G along the Cv contour}, we show that the real part of $G$ decreases sufficiently rapidly away from the critical point on the $\Cv{\pi/4}$ contour. The upper and lower halves of the $\Cv{\pi/4}$ contour are parametrized as
    \begin{equation}
        z(r) = \zcrit  + r \hat{e}_\pm, \quad r \geq 0
        \label{eq:parametrization of upper part of Cv contour}
    \end{equation}
    where $\hat{e}_\pm = -1 \pm \I$.  On both halves of the contour, the derivative of $G$ satisfies
      \begin{equation}\label{52}
            \frac{d}{dr} \Re( G(z(r)) - G(\zcrit) )
            \leq - \frac{2 r^2 }{( 1 + \zcrit + 2 r )^2}.
       \end{equation}
    This captures the cubic behavior of $G$ near the critical point, and the linear decay for large $r$: for some constants $C \AND r_0$ independent of $\zcrit$ and $N$,
    \begin{equation}
        \Re( G(z(r)) - G(\zcrit) ) \leq -H(r) := - 
        \begin{cases}
           C \zcrit^{-2} r^3 & r \leq r_0\\
           C \left( \zcrit^{-2} r_0^3 + \zcrit^{-2} (r - r_0) \right) & r > r_0 
        \end{cases}.
        \label{eq:estimate for G function along steep descent contour}
    \end{equation}

    \item \label{step:similar estimates along Cs contour}
        In \Lemref{lem:decay of G along the Cs contour}, we estimate $G$ on the $\Cw{\zcrit + \delta}$ contour. We use the parametrization
        \begin{equation}
            w(r) = \zcrit + r \hat{e}_\pm + \delta (\sigma N^{1/3})^{-1}, \quad r \geq 0,
            \label{eq:parametrization along Cw contour}
        \end{equation}
        where $\hat{e}_\pm = \pm \I$. Since $\Cw{\zcrit + \delta}$ is not a steep-descent contour, we cannot show that the derivative of $\Re(G)$ is strictly positive (c.f.~\Eqref{52}). So we first move the $\Cw{\zcrit + \delta}$ contour to a $\delta' < \delta$, and then locally align it with the Airy contours $\{e^{-\pi\I/3}\R^+ + \zcrit + \delta'\} \cup \{e^{\pi\I/3}\R^+ + \zcrit + \delta' \}$ for a very small distance $r' \leq r_0$. Then, we use the Taylor expansion and obtain an estimate analogous to the first equation in \eqref{eq:estimate for G function along steep descent contour} for $|w(r) - \zcrit| \leq r'$. For the rest of the contour, we show in \Lemref{lem:decay of G along the Cs contour}, we show  
        \begin{equation*}
            \Re(G(w(r)) - G(w(r_0)) \geq 0, \quad |w(r) - \zcrit| \geq r'.
        \end{equation*}

    \item \label{step:the residues go to zero} Finally, in \Lemref{lem:residues are small}, we estimate the contribution of the residues to the bound in~\Eqref{eq:exponential bound along contours for the kernel}. 
        The bound also shows that the residues vanish in the limit, and helps prove \Lemref{lem:pointwise convergence of loggamma kernel to airy kernel}. 
\end{enumerate}
\begin{oldnotes}
    The $G$ function looks concave when draw in mathematica. We tried to prove it then and failed. This is because we did not have the tools that we have now to analyze $G'$ for all $\theta$.
\end{oldnotes}

\begin{proof}[Proof of \Lemref{lem:pointwise convergence of loggamma kernel to airy kernel}]
    Using the estimates in steps $1$-$4$, we first show that the pointwise limit of the integral in~\Eqref{eq:log gamma kernel K for in position for asymptotic analysis} is the Airy kernel. We will use both the rescaled variables $\tilde{v},\tilde{v}'$ from \eqref{eq:rescaling the variables in the loggamma kernel to take a limit} and $v,v'$ in the following.

First consider the integral term in \eqref{eq:statement of pointwise convergence to airy kernel}, and split the integral over the top half of the contour into three parts
\begin{equation}
    \tilde\sigma^{-1} \tilde{K}^N(\tilde{v},\tilde{v}') = \int_{0}^{M \tilde{\sigma}^{-1}}
     + \int_{M\tilde{\sigma}^{-1}}^{r_0} + \int_{r_0}^{\infty}  \tilde\sigma^{-1} I(v,v',w(r)-v) dw(r) := I_1 + I_2 + I_3,
    \label{eq: splitting integral into proof of rescaled kernel convergence to airy}
\end{equation}
where $w(r)$ is the parametrization in~\Eqref{eq:parametrization along Cw contour} of the $\Cw{\zcrit}$ contour, and $M > 0$ is a parameter that will eventually go to infinity. Since the integrand is analytic in a tiny ball of size $M\tilde{\sigma}^{-1}$, we modify the contours so that they are locally aligned with the Airy contours when $r \leq M\tilde{\sigma}^{-1}$. 

We first estimate the absolute value of $I_3$. Note that there is a $C'$ such that $C' r \leq C r^3  - r$ for $r \geq r_0$. We will use this estimate in both the $v$ and $w$ variables.
\begin{align*}
      |I_3| 
      & = \left| -\frac1{2\pi\I}\int_{r_0}^{\infty} \frac{\pi\tilde{\sigma}^{-1}}{\sin(\pi(w-v))}
      e^{ N\left(G(v) - G(w)\right) + t \tilde\sigma (v-w)}\frac{dw}{w-v'} \right| \\
      & \leq C \delta^{-1} r_0^{-1} e^{-N H(|v|) + \tilde\sigma |v|} e^{-C' M^3} \int_{  r_0 }^\infty e^{-\pi r} d r , \\
      & \leq C \delta^{-2} r_0^{-1} e^{-C' |\tilde{v}|} e^{-C' M^3}, 
\end{align*}
where $H(r)$ is defined in \eqref{eq:estimate for G function along steep descent contour}, and we have used \Eqref{eq:order of magintude of tilde sigma}, $|\sin(\pi (w(r) - v))|^{-1} \leq c\delta^{-1} \tilde\sigma e^{-\pi r}, \AND |w(r) - v'|^{-1} \leq c r_0$ for some constant $c$.

To estimate $I_2$, we use $e^{-\pi r} \leq 1$, make the change of variable $\tilde r = \tilde \sigma r$, and use the bound $|w(r) - v'|^{-1} \leq c \delta^{-1} \sigma$; all the other estimates are the same as the ones used for $I_3$. Thus, we obtain:
\begin{align*}
    |I_2| & \leq  
    C \delta^{-2} e^{-C |\tilde{v}| } \int_{M}^{r_0 \tilde{\sigma}} e^{- \tilde{r}} d \tilde{r} \leq C \delta^{-2} e^{-C'|\tilde{v}|} e^{-M}.
\end{align*}
In $I_1$, we make the change of variable in \eqref{eq:rescaling the variables in the loggamma kernel to take a limit}, and take $N\to\infty$. By dominated convergence, the limit can be taken inside the integral, and by the argument in the first part of Section \ref{sec:estimates along the contours}, the first integral goes to the integrand of $\Kairy$ in the rescaled variables $(\tilde{v},\tilde{v}')$. Letting $M \to \infty$ shows that the integral term in \eqref{eq:log gamma kernel K for in position for asymptotic analysis} goes to $\Kairy$. From~\Lemref{lem:residues are small}, it follows that residues converge pointwise to $0$.
\end{proof}

Next, we flesh out the details of the estimates in steps 1 through 4.

\noindent\Stepref{step:G taylor expansion in intermediate regime}.
\textbf{Taylor expansion near the critical point.} We have already seen that the first two derivatives of $G$ (defined in \eqref{eq:u-G-f-in-intermediate-disorder}) vanish at the critical point. If the third and fourth derivatives of $G$ were well-behaved, the Taylor expansion for $G(z)$ is an effective approximation when $z$ is close to the critical point:
\[
    G(z) - G(\zcrit) = \frac{G^{(3)}}{3!}(\zcrit) (z - \zcrit)^3 + \frac{G^{(4)}}{4!}(\xi) (z - \zcrit)^4,
\]
for $\xi \in \{ y \colon |y-\zcrit| < |z-\zcrit| \}$. Since $G$ is an analytic function, it is clear that $G^{(3)}$ and $G^{(4)}$ are well-behaved for fixed $\zcrit$. However, we allow $\zcrit \to \infty$; ~\Propref{prop:estimates for third and fourth derivative of G close to critical point} shows roughly that $\frac{G^{(4)}(z)}{G^{(3)}(\zcrit)} \approx \frac{C}{\zcrit}$ for a constant $C > 0$,
when $|z(r) - \zcrit| \leq \zcrit/2$ and $\zcrit \to \infty$.
\begin{prop}
    
    When $|z - \zcrit| \leq \zcrit/2$, 
    \begin{align}
        \frac{2}{(2 + \zcrit)^2} \leq -G^{(3)}(\zcrit) & - \frac{4}{\zcrit^3} \leq  \frac{2}{\zcrit^2}, \label{pppg}\\
        \left|G^{(4)}(z) \right| & \leq \frac{96}{\zcrit^4} + \frac{32}{\zcrit^3}.\label{ppp|}
    \end{align}
    \label{prop:estimates for third and fourth derivative of G close to critical point}
\end{prop}
\begin{proof}[Proof of~\Propref{prop:estimates for third and fourth derivative of G close to critical point}]
 Recall the series expansion for the digamma function in \eqref{eq:series representation for digamma function}. Differentiating~\Eqref{eq:G function second definition} thrice we obtain
   \begin{equation*}
        -G^{(3)}(\zcrit)  = - 2 \Psi^{(2)}(\zcrit)
        = \frac{4}{\zcrit^3} + 4 \sum_{n=1}^{\infty} \frac{1}{(n+ \zcrit)^3}.
    \end{equation*}
    Estimating the sum with an integral, we get
    \begin{align*}
     4 \int_1^{\infty} \frac{1}{(x + \zcrit)^3} dx \leq -G^{(3)}(\zcrit) - \frac{4}{\zcrit^3} & \leq  4 \int_0^{\infty} \frac{1}{(x + \zcrit)^3} dx
    \end{align*}
    which proves \eqref{pppg}. We estimate $G^{(4)}(z)$ similarly: To apply the integral test as before, we first show that $|x + z|$ is increasing with $x \in \R^+$. It is clear that if $|z - \zcrit| \leq \zcrit/2$, then $z$ has positive real part and consequently, so does $x+z$ for all $x > 0$. It follows that $|x + z|$ increases with $x$. Then, using~\Eqref{eq:series representation for digamma function} and $|x + z| \geq 2^{-1}(x + \zcrit)$,
    \begin{align*}
        \left|\Psi^{(3)}(z) \right|
        & \leq \sum_{n=0}^{\infty} \frac{6}{|n+z|^4} \leq \frac{96}{|\zcrit|^4} + \int_0^{\infty} \frac{96}{|x + \zcrit|^4} dx,
    \end{align*}
    which proves \eqref{ppp|}.
    
\end{proof}

\noindent\Stepref{step:G is a decreasing along the contour}.~\textbf{Decay of $G$ along the $\Cv{\pi/4}$ contour.}
 In the following lemma, we first compute the derivative of $G$ along a general contour. We will use this computation repeatedly to estimate $G$ along the $\Cv{\pi/4}$ and $\Cw{\zcrit}$ contours (\Lemref{lem:decay of G along the Cs contour}) and to estimate the residues in~\Lemref{lem:residues are small}.

  \begin{lemma}
      Let $z(r) = \zcrit + v(r)$ be a contour. Then, the derivative of  $\Re(G)$ is
      \[
        \frac{d}{dr} \Re(G(z(r))) = 2 \sum_{n=0}^{\infty} \frac{ - \Re( v'(r) v(r)^2 ) (n+\zcrit)^2 + \Re( v'(r) ) |v(r)|^4 }{(n+\zcrit)|(n+\zcrit)^2 - v(r)^2|^2}
      \]
      
      \label{lem:derivative of G general calculation}
  \end{lemma}
  \begin{proof}
      From~\Eqref{eq:G function second definition},
      \[
          \frac{d}{dr} G(z(r))
          = z'(r) \left( \Psi(\zcrit + v(r) ) - \Psi(\zcrit) \right) + z'(r) \left( \Psi(\zcrit - v(r)) - \Psi(\zcrit) \right).
      \]
    
    Using~\Eqref{eq:series representation for digamma function},
    \begin{align*}
        \frac{d}{dr} G(z(r))
        & = z'(r) \sum_{n=0}^{\infty}  \frac{2}{n+\zcrit} - \frac{1}{n+z(r)} -  \frac{1}{n+ \theta - z(r)} \\
        & = v'(r) \sum_{n=0}^{\infty}  \frac{2}{n+\zcrit} - \frac{2(n + \zcrit)}{(n+\zcrit + v(r))(n+\zcrit - v(r))}\\
        & = 2 \sum_{n=0}^{\infty} \frac{ - v'(r) v^2 (n+\zcrit)^2 + v'(r) |v(r)|^4 }{(n+\zcrit)|(n+\zcrit)^2 - v^2|^2}.
    \end{align*}
\end{proof}

\begin{lemma}
    $\Re(G)$ in~\Eqref{eq:G function second definition} satisfies the following derivative bound:
    \[
        \frac{d}{dr} \Re ( G(z(r))) \leq - \frac{2 r^2}{( 1 + \zcrit + 2r  )^2}, \quad r > 0, 
    \]
    where $z(r)$ is the parametrization of $\Cv{\pi/4}$ given in~\Eqref{eq:parametrization of upper part of Cv contour}.
    \label{lem: decay of G along the Cv contour}
\end{lemma}

\begin{proof}
    We parametrize the upper-half of the $\Cv{\pi/4}$ contour as in~\Lemref{lem:derivative of G general calculation} with $v(r) = r \hat{e}$ where $\hat{e} = -1 + \I$. Then,
    \begin{align*}
        \frac{d}{dr} \Re(G(z(r)))
        & = 2 \sum_{n=0}^{\infty} \frac{ - 2 r^2 (n+\zcrit)^2 - 4 r^4}{(n+\zcrit)|(n+\zcrit)^2 - 2 r^2 \I|^2}\\
        & \leq -4 \sum_{n=0}^{\infty} \frac{ r^2 }{(n+\zcrit + 2 r)^3} \leq - 4 \int_{1}^{\infty}  \frac{r^2}{( x + \zcrit  + 2 r )^3} \, dx \\
        & = - 2 \frac{r^2}{( 1 + \zcrit  + 2 r )^2}.
    \end{align*}

    This captures the behavior of $G$ along the steep-descent contours rather well:  cubic near the critical point, and then linear decay for $r \geq C\zcrit$. $G$ behaves symmetrically in the lower half plane, and hence satisfies the same estimates on the lower half of the $\Cv{\pi/4}$ contour.
    
\end{proof}

\noindent\Stepref{step:similar estimates along Cs contour}. \textbf{Decay along the $\Cw{\zcrit}$ contour.}
\begin{lemma}
    $\Re(G)$ in~\Eqref{eq:G function second definition} increases away from the critical point along the $\Cw{\zcrit}$ contour.
    \label{lem:decay of G along the Cs contour}
\end{lemma}
\begin{proof}
    Recall that the $\Cw{\zcrit}$ contour starts off at a distance $\hat{\delta}_N = \delta \tilde\sigma^{-1}$ away from $\zcrit$. Let $w(r)$ be  the parametrization of $\Cw{\zcrit}$ in~\Eqref{eq:parametrization along Cw contour}; as before we focus on the upper half of the contour. Using \Lemref{lem:derivative of G general calculation} and $v(r) = r \I + \hat{\delta}_N$, we get
    \begin{equation*}
        \frac{d}{dr} \Re(G(w(r)))
        = 2 \sum_{n=0}^{\infty} \frac{ - \Re\left( \I (\hat{\delta}_N - r^2 + 2 \hat{\delta}_N r \I ) \right) (n+\zcrit)^2 + \Re( \I ) |v|^4 }{(n+\zcrit)|(n+\zcrit)^2 - v^2|^2}
        \geq 0.
    \end{equation*}

\end{proof}

\par\noindent\Stepref{step:the residues go to zero}. \textbf{Triviality of the residues}
\begin{lemma}  There exist constants $c_1,C>0$ independent of $N \AND \zcrit$ such that for $\i= 1,\ldots \lfloor |\Im(v)| \rfloor$,
    the residues in~\Eqref{eq:log gamma kernel K for in position for asymptotic analysis} satisfy 
    \[
        \log \left|\Res_{\i}(v,v')\right| 
        \leq  
        \begingroup
        \left\{
        \renewcommand*{\arraystretch}{2}
        \begin{array}{cc}
            - c_1 N \i^2 \dfrac{|\Im(v)|}{\zcrit^2} & 1 \leq |\Im(v)| \leq C \zcrit \\
            - c_1 N \i \log \left( 1 + \dfrac{|\Im(v)|}{\zcrit} \right) & |\Im(v)| > C \zcrit
        \end{array}
        \right.
        \endgroup
    \]
    when $v, v' \in \Cv{\pi/4}$. If $C  \zcrit < 1$, the first bound holds vacuously.

    \label{lem:residues are small}
\end{lemma}
Lemma~\ref{lem:residues are small} helps show the estimate on the kernel in~\Eqref{eq:exponential bound along contours for the kernel} that's used in the Hadamard bound in \Lemref{lem:rigorous hadamard bound on steepest descent contours}. It also shows that the residues go to $0$ as $\tilde\sigma \to \infty$.
\begin{proof}
    From~\Eqref{eq:general residues of the sine in our fredholm determinant formula} and \Eqref{eq:u definition in intermediate disorder}, we can write the residues in the form
    \begin{equation}
        \begin{aligned}
            |\Res_{\i}(v,v')|
            & = \left| \left( \frac{\Gamma(v)}{\Gamma(v+\i)} \frac{\Gamma(\theta - v - \i)}{\Gamma(\theta - v)} \right)^N \frac{e^{2 \Psi(\zcrit) N \i - t \i \tilde\sigma  }}{v + \i -v'} \right| \\
            & \leq e^{ N(G(v) - G(v+\i)) + |t| \i \tilde\sigma }
        \end{aligned}
        \label{eq:estimate for the residues along the contours}
    \end{equation}
    since $\i \geq 1$. We estimate $G(v) - G(v+\i)$ using~\Lemref{lem:derivative of G general calculation} again. Let $v(r) = k \hat{e}_+ + r$ in the parametrization of the contour in~\Lemref{lem:derivative of G general calculation}, where $\hat{e}_+ = -1 + \I$. Since $k = |\Im(v)|$, \Eqref{eq:number of residues and zcrit definition} implies that we only need to consider $r$ in the range $0 \leq r \leq k$ (assuming $\delta$ is small enough). We interpolate between $G(v)$ and $G(v+l)$ by computing the following derivative:
    \begin{align*}
        \frac{d}{dr} \Re(G(z(r))) 
        & = 2 \sum_{n=0}^{\infty} \frac{ r(2k - r)  (n+\zcrit)^2 +  ((r - k)^2 + k^2)^2  }{(n+\zcrit)|(n+\zcrit) + v|^2|(n+\zcrit) - v|^2} \\
        & \geq 2 \sum_{n=0}^{\infty} \frac{ r (2k - r) (n+\zcrit)^2 +  ((r - k)^2 + k^2)^2  }{(n+\zcrit)(n + \zcrit + 2k-r)^4}\\
        & \geq 2 \int_{1 + \max(k,\zcrit)}^{\infty} \frac{ r (2k - r) x }{(x + 2k-r)^4}\,dx + 2 \int_{1 + \zcrit}^{\infty}  \frac{ ((r - k)^2 + k^2)^2 }{x (x + 2k-r)^4} \,dx := I_1 + I_2.
    \end{align*}
    The limits of integration in $I_1$ have been chosen as follows. Notice that $f(x) = x/(x + 2k - r)^{4}$ is decreasing for $x \geq \max(k,\zcrit)$. This lets us use the integral test to estimate the sum.
    
    Our bounds on the integrals $I_j,\,j=1,2$ will ensure:
    \begin{numcases}{G(v+\i) - G(v) \geq}
    c_1 \i^2 \dfrac{|\Im(v)|}{\zcrit^2} & $|\Im(v)| \leq C \zcrit$ \label{eq:g function on residues case 1}\\
    c_1 \i \log \left( 1 + \dfrac{|\Im(v)|}{\zcrit} \right) & $|\Im(v)| > C \zcrit$ \label{eq:g function on residues case 2}
    \end{numcases}
    \begin{enumerate}
        \item The first bound \Eqref{eq:g function on residues case 1} ensures that the exponent in~\Eqref{eq:estimate for the residues along the contours} contains a negative term of order at least $N \zcrit^{-2}$, and thus overwhelms $\tilde\sigma = \Theta( (N \zcrit^{-2})^{1/3})$ as they both go to infinity (by \eqref{lbonbeta}).
        \item The second bound \Eqref{eq:g function on residues case 2} ensures that the residue is integrable in $v$ on the contour $\Cv{\pi/4}$ over the range $|\Im(v)| \in [C\zcrit,\infty)$.  
    \end{enumerate}
    
    Explicitly computing $I_1$, we get
    \begin{equation*}
        \begin{aligned}
            I_1
            & = \frac{ r (2k - r)(3(1 + \max(k,\zcrit)) + 2k -r) }{3 (1 + \max(k,\zcrit) + 2k-r)^3} \geq \frac{ r k}{3 (1 + \max(k,\zcrit) + 2k)^2},
        \end{aligned}
    \end{equation*}
    for $r \leq k$. For the second integral, 
    \begin{equation*}
        \begin{aligned}
            I_2
            & \geq k^4 \left( \frac{1}{(2k-r)^4} \log \left( 1 + \frac{2k - r}{1 + \zcrit} \right) - \frac{2}{(2k-r)^3 (1 + \zcrit + (2k-r))} \right),\\
        \end{aligned}
    \end{equation*}
    where we have used the elementary integral
    \[
        \int_{a}^{\infty} \frac{dx}{x(x+c)^4 }=-\frac{6 a^2+15 a c+11 c^2}{6 c^3 (a+c)^3} +  \frac1{c^4} \log \left( 1 + \frac{c}{a} \right).
    \]
    Here, for $k \geq C \zcrit$ where $C$ is some constant, the first term in $I_2$ dominates the second for all $r \leq k$. 
    
    Therefore, integrating over $r$, we get \eqref{eq:g function on residues case 1} and \eqref{eq:g function on residues case 2} for some constants $c_1,C$.

\end{proof}
Finally, we prove the inequality used in the Hadamard bound in~\Eqref{eq:exponential bound along contours for the kernel}. 
\begin{lemma}
    \label{lem:rigorous hadamard bound on steepest descent contours}
    There exist constants $c_1,c_2,C > 0$ that are independent of $N$ and $\zcrit$ such that for all $N$ large enough,
    \begin{align}
        |\Ksteepestdescent(v,v')| & \leq 
        \begingroup
        \left\{
        \renewcommand*{\arraystretch}{2.5}
        \begin{array}{cl}
            c_1 \tilde{\sigma} \exp\left( - c_2 N \zcrit^{-2} |\Im(v)|^3  \right)  &   |\Im(v)| < 1 \\
            c_1 \tilde{\sigma} \exp\left( - c_2 {N \zcrit^{-2} |\Im(v)|} \right) & 1 \leq |\Im(v)| \leq C \zcrit \\
            c_1 \left(1 + \dfrac{|\Im(v)|}{\zcrit}\right)^{-c_2 N}  & |\Im(v)| > C \zcrit
        \end{array}
        \right.
        \endgroup \nonumber \\
        & =: f(v,N)
        \label{eq:definition of f v N in hadamard bound}
    \end{align}
    on the contour $\Cv{\pi/4}$. Consequently the $m$\textsuperscript{th} term of the Fredholm series for $\Ksteepestdescent$ in~\Eqref{eq:fredholm series expansion} satisfies
    \begin{equation}
        \frac{1}{m!} \int_{\Cv{\pi/4}} dv_1 \cdots \int_{\Cv{\pi/4}} dv_m \det(\Ksteepestdescent(v_i,v_j))_{1\leq i,j\leq m} \leq \frac{C}{m^{(m-1)/2}}.
        \label{eq:rigorous bound for mth term in fredholm series} 
    \end{equation}
\end{lemma}
\begin{proof}
    There are many undetermined constants in this proof, and we allow them to change from line to line. Using the technique of splitting up the integral term as in \eqref{eq: splitting integral into proof of rescaled kernel convergence to airy}, and from the estimates in~\Lemref{lem: decay of G along the Cv contour} and~\Lemref{lem:decay of G along the Cs contour}, it follows that the integral in~\Eqref{eq:log gamma kernel K for in position for asymptotic analysis} has two regimes of behavior: for constants $c_1,c_2,C > 0$
    \[
        \int_{\Cw{\zcrit}} I(v,v',\w) dw \leq
        \begingroup
        \left\{
        \renewcommand*{\arraystretch}{2.5}
        \begin{array}{cl}
            c_1 \tilde{\sigma} \exp\left( - c_2 N \zcrit^{-2} |\Im(v)|^3  \right)  &   |\Im(v)| \leq C \zcrit \\
            c_1 \tilde{\sigma} \exp\left( - c_2 N |\Im(v)|  \right)  &   |\Im(v)| > C \zcrit \\
        \end{array}
        \right.
        \endgroup
    \]
    The $\Cw{\zcrit}$ contour is a small distance $\tilde\delta_N = \delta (\sigma N^{1/3})^{-1}$ away from $\zcrit$. From \eqref{eq:number of residues and zcrit definition}, it follows there are about $|\Im(v)| + \tilde\delta_N$ residues. Hence, when $|\Im(v)| < 1$ and when $N$ large enough, there are no residues. We estimate the contribution of the residues when $\Im(v) > 1$. For $N$ large enough, and $1 \leq |\Im(v)| \leq C \zcrit$,~\Lemref{lem:residues are small} implies
    \begin{align*}
        \sum_{\i=1}^{\lfloor|\Im(v)|\rfloor} \Res_{\i}(v,v') 
        & \leq \sum_{\i=1}^{\lfloor|\Im(v)|\rfloor} c_1 \exp\left(  - c_1 N \i \dfrac{|\Im(v)|}{\zcrit^2} \right) \leq c_1' \exp\left( - c_2 N \zcrit^{-2} |\Im(v)| \right).
    \end{align*}
    When $|\Im(v)| > C \zcrit$, 
    \begin{align*}
        \sum_{\i=1}^{\lfloor|\Im(v)|\rfloor} \Res_{\i}(v,v') 
        & \leq \sum_{\i=1}^{\lfloor|\Im(v)|\rfloor} c_1 \left(1 + \dfrac{|\Im(v)|}{\zcrit}\right)^{-c_2 Nl}
        \leq c_1' \frac{1}{(1 + |\Im(v)|/\zcrit)^{c_2 N}}.
    \end{align*}
    
    Thus, \eqref{eq:definition of f v N in hadamard bound} follows, and we can integrate the bound over the $\Cv{\pi/4}$ contour to obtain 
    \begin{align*}
        \int_{\Cv{\pi/4}} f(v,N) dv 
        & \leq C \left( \tilde{\sigma} \int_0^{1} e^{-c_2 N \zcrit^{-2} x^3}\,dx  
        + \tilde{\sigma} \int_1^{C \zcrit} e^{-c_2 N \zcrit^{-2} x}\,dx 
        + \int_{C\zcrit}^{\infty} \left(1 + \dfrac{x}{\zcrit}\right)^{-c_2 N}\right) \\
        & \leq C \left( \frac{\tilde{\sigma}}{(N \zcrit^{-2})^{1/3}}  
        + \tilde{\sigma} e^{ - c_2 N \zcrit^{-2} } +  \frac{(1+C)^{-c_2 N}}{N\zcrit^{-1}}\right).
    \end{align*}
    Since $\tilde\sigma = \Theta((N\zcrit^{-2})^{1/3})$, it is clear that the integral is bounded above by a constant independent of $N$ and $\zcrit$. The Hadamard inequality now implies~\Eqref{eq:rigorous bound for mth term in fredholm series}.
    

\end{proof}

\appendix
\section{Fredholm determinant  as a limit of the  formula of Borodin-Corwin-Ferrari-Veto}
\label{sec:appendix on fredholm determinant}
\subsection{The BCFV theorem}
\citet{MR3366125} consider a mixed polymer that consists of Sepp\"al\"ainen's \loggamma{} polymer~\citep{MR2917766} and the O'Connell-Yor semi-discrete polymer~\citep{MR2952082}. For $N \geq 1$, the up-right paths $\path$ consist of a discrete portion $\path^d$ adjoined to a semi-discrete portion $\path^{sd}$. The discrete portion is an up-right nearest-neighbor path on $\Z^2$ that goes from $(-N,1)$ to $(-1,n)$ for some $1\leq n\leq N$. The semi-discrete path goes from $(0,n)$ to $(\tau,N)$: For $0 \leq s_n < \cdots < s_{N-1} \leq \tau$, it consists of horizontal segments on $(s_i,s_{i+1})$ for $i=n,\ldots,N-2$ and a final interval $(s_{N-1},\tau)$ connected by vertical jumps of size $1$ at each $s_i$. For $1\leq m,n\leq N$ let $\wt_{-m,n}$ be independent \expgamma{} random variables with parameter $\theta$, and for all $1\leq n\leq N$ let $B_n$ be independent Brownian motions. The paths have energy
\begin{equation*}\label{eqEnergy}
    \begin{aligned}
        H_{\beta}(\path)
        &= -\sum_{(i,j)\in \path^d} \wt_{i,j} + B_n(s_n)+\big(B_{n+1}(s_{n+1})-B_{n+1}(s_n)\big)+\ldots+\big(B_N(\tau)-B_N(s_{N-1})\big).
    \end{aligned}
\end{equation*}
The partition function is given by
$$
\mathbf{Z}^{N}(\tau) = \sum_{i=1}^{N} \sum_{\path^d:(-N,1)\nearrow (-1,i)} \, \int_{\path^{sd}:(0,i)\nearrow(\tau,N)} e^{-H_{\beta}(\path)} d\path^{sd}
$$
where $d\path^{sd}$ represents the Lebesgue measure on the simplex $0\leq s_n<s_{n+1}<\cdots <s_{N-1}\leq \tau$.

\begin{remark}
   When $\tau = 0$, the polymer is simply the standard point-to-point \loggamma{} polymer. There is no semi-discrete part, and the discrete path is forced to end at $(-1,N)$. 
\end{remark}

\begin{theorem}[\citet{MR3116323}, Theorem 2.1]
    \label{thm:BCFV formula}
    Fix $N\geq 9$, $\tau > 0$ and $\theta>0$. For all $u\in \C$ with positive real part
    \begin{equation}
        \EE\left[ e^{-u \mathbf{Z}^{N}(\tau)} \right] = \det(\Id+ \Kbcfv)_{L^2(\Cv{\varphi})}
        \label{eq:laplace transform equation in bcfv theorem}
    \end{equation}
    where
    \begin{equation}
        \begin{aligned}
            \Kbcfv(v,v')
            & = \frac{1}{2\pi \I}\int_{\Cs{v}}\, \frac{1}{\sin(\pi s)} \left( \frac{\Gamma(v)}{\Gamma(s+v)} \frac{\Gamma(\theta-v-s)}{\Gamma(\theta-v)} \right)^N \frac{e^{\tau(s v + s^2/{2})}}{v+s - v'} u^s \, ds\\
            & =: \frac{1}{2\pi \I}\int_{\Cs{v}}I_{u,\tau}(v,v',s) \, d s\, .
        \end{aligned}
        \label{eq:BCFV kernel with tau in it}
    \end{equation}
\end{theorem}
The $\Cv{\varphi}$ contour is the wedge-shaped contour defined in~\Thmref{thm:our version of BCFV formula with tau equals 0}, and depends on the angle $\varphi$ and the parameter $\theta$. The $\Cs{v}$ contour depends on $v$ and parameters $R \AND d$. For every \mbox{$v\in \Cv{\varphi}$} we choose $R=-\Re(v)+3\theta/4$, $d>0$, and let $\Cs{v}$ consist of straight lines from $R-\I \infty$ to $R-\I d$ to $\theta/8-\I d$ to $\theta/8+\I d$ to $R+\I d$ to $R+\I\infty$. The parameter $d$ must be taken small enough so that $v+\Cs{v}$ does not intersect $\Cv{\varphi}$. Both contours are oriented to have increasing imaginary part.

\begin{remark}
    In \citet{MR3366125}, the $\Cs{v}$ contour consisting of straight lines from $R-\I \infty$ to $R-\I d$ to $1/2-\I d$ to $1/2+\I d$ to $R+\I d$. The formula only holds if the poles in $s$ of $\Gamma(\theta-v-s)$ lie strictly to the right of the 
    contour $\Cs{v}$, and this imposes a lower bound $\theta>1$ that wasn't noticed by them. We remove the restriction $\theta>1$ as follows. 
    
    Note first of all that both sides of \eqref{eq:laplace transform equation in bcfv theorem}, with $\Cs{v}= \Cs{v,1/2}$ are analytic functions of $\theta$ in some region of $\mathbb{C}$
    containing the ray $(1,\infty)$, on which they coincide, by \citet{MR3366125}: The left hand side is actually analytic in a region containing the ray $\theta\in (0,\infty)$ because the expectation is an
    $N^2$-fold integral of a function $e^{-u \mathbf{Z}^{N,N}(\tau)}$ of the variables $\xi_{i,j}$, $i,j=1,\ldots,N$ with 
    \begin{align*}
       \EE[ e^{-u \mathbf{Z}^{N}(\tau)} ] 
                & =  \int F( \xi_{ij}) \prod_{i,j} \frac{e^{-\wt_{i,j}} \wt_{i,j}^{\theta - 1}}{\Gamma(\theta)} d\wt_{i,j} \qquad F(\xi_{ij})=  \EE[ e^{-u \mathbf{Z}^{N}(\tau)}~|~\xi_{i,j}].
    \end{align*}

    The right hand side is analytic because the Fredholm determinant of a kernel analytic in $\theta$ is analytic in $\theta$, as long as one has a uniform bound for the kernel in that region. 
    
    Call $\Cs{v,\eta}$ the contour consisting of straight lines from $R-\I \infty$ to $R-\I d$ to $\eta-\I d$ to $\eta+\I d$ to $R+\I d$. Fix $\theta^*>0$. We can use Cauchy's theorem to deform the contour in \eqref{eq:BCFV kernel with tau in it} to $\Cs{v,\theta^*/4}$ without changing the kernel, since the only obstacle is the zero of the sine in the denominator,
    which is at the origin.  Proposition \ref{prop:our estimate for the BCFV kernel for the Hadamard bound} gives us a uniform bound on the kernel in the region, say $[\theta^*/16, 2]$, that can be used together with Hadamard's bound to control the Fredholm series. Using this new representation of the kernel, the determinant in \eqref{eq:laplace transform equation in bcfv theorem} is an analytic function of $\theta$, now in a region containing $[\theta^*-\gamma, 1+\gamma]$ for some $\gamma>0$.  Because the kernel is
    unchanged, this coincides with the old determinant for $\theta\in (1,1+\gamma)$. By the
    formula from \citet{MR3366125}, the determinant coincides with the left hand side of \eqref{eq:laplace transform equation in bcfv theorem} for $\theta\in (1,1+\gamma)$.  But the 
    left hand side is analytic in a region containing the ray $(0,\infty)$, hence the determinant with the new, deformed kernel is equal to the left hand side on a region containing the interval $[\theta^*-\gamma, 1+\gamma]$, including at the value $\theta^*$.
    
    \label{rem:the error in the bcfv contour}
\end{remark}

\begin{remark}
    We write $\Cs{v}$ contour as a union of the vertical contour $\Cw{-Re(v) + R}$ defined in~\Thmref{thm:our version of BCFV formula with tau equals 0} and the ``sausage'' $\Cs{v,\Box}$ (which consists of $(\Cs{v} \setminus \Cw{-Re(v) + R}) \cup [R+\I d, R - \I d]$ and forms a clockwise loop). The integral of the kernel $\Kbcfv$ over $\Cs{v,\Box}$ consists of residues due to the sine that we will estimate separately. For each $v$, there is some wiggle room in the $R$ parameter that allows the vertical contour $\Cw{-\Re(v) + R}$ to avoid the singularity of the sine function in $I_{u,\tau}(v,v',s)$.
\end{remark}

\subsection{Proof of Theorem \ref{theorem4.1}}
In this section, we obtain Theorem \ref{theorem4.1} by letting $\tau \searrow 0$ in~\Thmref{thm:BCFV formula}.
We may do so if we can truncate the series that defines the Fredholm determinant of $\mathbf{Z}^{N}(\tau)$ uniformly in $\tau$. This is done by proving an estimate on $\Kbcfv(v,v')$ that depends favorably with $\tau$, and then using Hadamard's inequality in the usual way (see~\eqref{eq:rigorous bound for mth term in fredholm series}). The constants in the following propositions may depend on the angle of the contour $\varphi \in (0,\pi/4]$.
\begin{prop}[BCFV kernel estimate]
    When $|\Im(v)| > \max((2e^{\tau \theta/2}|u|)^{1/2N},c_3 \theta))$, for some constants $c_1,c_2 , c_3 > 0$, the kernel in~\Eqref{eq:BCFV kernel with tau in it} satisfies the bound
    \[
        |\Kbcfv(v,v')| \leq 2\frac{e^{\tau \theta/2} |u|}{|\Im(v)|^{2N}} e^{ - c_1 \tau |\Im(v)|} + e^{C(\theta,N,\tau)} e^{ - c_2 N |\Im(v)|(\log |v| - c_3 \theta)},
    \]
    where $|C(\theta,N,\tau,\varphi)| \leq  C'(\theta^*,\varphi)( N (\theta + |\log \theta| + \theta^{-1} + 1) + \tau \theta)$ for $\theta \geq \theta^*$. Here $\theta^* > 0$ is any fixed number, and $C'$ depends only on $\theta^*$ and the angle of the $\Cv{\varphi}$ contour. 
   \label{prop:our estimate for the BCFV kernel for the Hadamard bound}
\end{prop}

Proposition~\ref{prop:our estimate for the BCFV kernel for the Hadamard bound} is proved by estimating $I_u$ on the closed rectangular contour $\Cs{v,\Box}$, and the vertical line $\Cw{R}$.

\begin{prop}[Integral over the $\Cs{v,\Box}$ contour]
    When $|\Im(v)| > \max( (2e^{\tau \theta/2}|u|)^{{\frac1{2N}}}, c_3 \theta)$,
   \[
       \left| \int_{\Cs{v,\Box}} I_{u,\tau}(v,v',s) \, ds \right|  \leq \frac{e^{\tau \theta/2} |u|}{|\Im(v)|^{2N}} e^{ - c \tau |\Im(v)|},
   \]
   where $c > 0$ is some constant, and $c_3$ is the same constant that appears in \Propref{prop:our estimate for the BCFV kernel for the Hadamard bound}.
   \label{prop:integral-of-bcfv-kernel-over-box}
\end{prop}

\begin{prop}[Integral over the $\Cw{R}$ contour]
   For $|\Im(v)| \geq c_3 \theta$,
   \[
       \int_{\Cw{R}} I_{u,\tau}(v,v',s) \,ds \leq e^{C(\theta,N,\tau)} e^{ - c_1 N |\Im(v)|(\log |v| - c_2  \theta)},
   \]
   where $c_1,c_2,c_3$ and $C(\theta,N,\tau)$ are the same constants that appear in~\Propref{prop:our estimate for the BCFV kernel for the Hadamard bound}.
   \label{prop:integral-of-bcfv-kernel-over-vertical-line}
\end{prop}

Before proving the propositions, we complete the proof of Theorem \ref{theorem4.1}.
\begin{proof}[Proof of Theorem \ref{theorem4.1}]
    In \Thmref{thm:BCFV formula}, the vertical contour $\Cw{R}$ is at $R = -\Re(v) + 3\theta/4$. This must be moved to the critical point so that $R = -\Re(v) + \theta/2 + \delta$ for small $\delta$. This is done by using the bound on $I_{u,\tau}(v,v',s)$ in~\Eqref{eq:estimate on integrand of BCFV kernel along vertical line}, and hence we can truncate the vertical contour at large $|\Im(s)|$ and use Cauchy's theorem to move over the vertical contour. 

    Next, we have to take a limit $\tau \searrow 0$ in \Eqref{eq:laplace transform equation in bcfv theorem}. Since $u$ has positive real part, $e^{-u\partfunc}$ is absolutely bounded, and we can take the limit inside the integral by bounded convergence. For the right hand side, we can use the Hadamard inequality argument in~\Secref{sec:estimates along the contours} and the bounds in~\Propref{prop:integral-of-bcfv-kernel-over-box} and \Propref{prop:integral-of-bcfv-kernel-over-vertical-line} to show that it converges to the Fredholm determinant of the pointwise limit of the kernel $\Kbcfv$ as $\tau \searrow 0$. 
\end{proof}

\begin{proof}[Proof of Proposition~\ref{prop:integral-of-bcfv-kernel-over-box}]
    The integral over $\Cs{v,\Box}$ simply collects residues from the poles of $\sin^{-1}(\pi s)$. Then,
\begin{align*}
    \int_{\Cs{v,\Box}} I_{u,\tau}(v,v',s) \, ds
    & = \sum_{i=1}^{q(v)}  \left(\frac{\Gamma(v)\Gamma(\theta - v - i)}{\Gamma(v+i)\Gamma(\theta - v)} \right)^N u^i 
    \frac{e^{\tau(\Re(v) i + i^2/2)}}{|v + i - v'|} (-1)^i \\ 
    & = \sum_{i=1}^{q(v)} \Res_{u,i}(v,v'),
\end{align*}
where $q(v) \leq R$ is the number of zeros of the sine caught inside the sausage (\ref{eq:number of residues and zcrit definition}). 
Since $v,v' \in \Cv{\varphi}$, we have
\begin{equation}
    \begin{aligned}
        \Re(v) & = \frac{\theta}{2} - \cot(\varphi) |\Im(v)|.
    \end{aligned}
    \label{eq:v on the contour Cv}
\end{equation}
Then, we may estimate $q(v)$ as follows:
\begin{align*}
    q(v) \leq R & = -\Re(v) + \frac{\theta}{2} + \delta  = \cot(\varphi) |\Im(v)| + \delta,
\end{align*}
where $0 < \delta \leq \frac{\theta}{4}$. For our bound, the number of residues doesn't matter, and the contribution of the first residue dominates. The ratio of gamma functions in the residues become $\Gamma(v)/\Gamma(v+i) = \prod_{j=0}^{i-1} (v+j)^{-1}$ and $\Gamma(\theta -v - i)/\Gamma(\theta - v) = \prod_{j=1}^{i} (\theta - v-j)^{-1}$. It is clear that $|v + i| \geq |\Im(v)|$ and $|\theta -v - i| \geq |\Im(v)|$. The $|v - v' + i|^{-1}$ term can be bounded above by a constant, since $|\operatorname{arg}(v - v')| > \varphi$. 
Therefore for $|\Im(v)| > (e^{\tau \theta/2}|u|/2)^{{\frac1{2N}}}$,
\begin{equation*}
    \begin{aligned}
        \sum_{i =1}^{q(v)} \left| \Res_{u,i}(v,v') \right|
        & \leq \sum_{i =1}^{q(v)} \frac{1}{|\Im(v)|^{2Ni}} |u|^i e^{i \tau \theta/2} e^{ - \tau i (\cot(\varphi) |\Im(v)| - i/2)}\\
        & \leq \sum_{i =1}^{q(v)} \left(\frac{e^{\tau \theta/2} |u|}{|\Im(v)|^{2N}}\right)^i e^{ - \tau c |\Im(v)| }, \\
        & \leq 2 \frac{e^{\tau \theta/2} |u|}{|\Im(v)|^{2N}} e^{ - \tau c |\Im(v)| },
    \end{aligned}
    \label{eq:estimate for sum of residues when setting tau to 0}
\end{equation*}
where $c$ is a $\varphi$-dependent constant that comes from bounding the $i(\cot(\varphi) |\Im(v)| - i/2)$ term on the interval $1 \leq i \leq \cot(\varphi) |\Im(v)| + \delta$. We choose the constant $c_3$ to ensure that $\cot(\varphi) |\Im(v)| \geq \cot(\varphi) c_3 \theta > \theta/4 \geq \delta$. The same constant $c_3$ appears in \Propref{prop:integral-of-bcfv-kernel-over-vertical-line}.
\end{proof}

\begin{proof}[Proof of ~\Propref{prop:integral-of-bcfv-kernel-over-vertical-line}]
    We will focus first on estimating the product of gamma functions in $I_{u,\tau}(v,v',s)$. For $s \in \Cw{-\Re(v) + R}$,
    \begin{equation}
        \Re(s) = \delta + \cot(\varphi) |\Im(v)|.
        \label{eq:s on the contour Cs}
    \end{equation}
    Stirling's formula holds whenever $\arg(z)$ remains bounded away from $\pm \pi$ \citep[6.1.41]{MR0167642},
    $$
    \log\Gamma(z) =  (z-\tfrac12) \log z - z + \tfrac12 \log 2\pi + \Ord \left(\frac1{|z|} \right),
    $$
    and
    $$
    \Re(\log\Gamma(z)) = -\Im(z) \arg(z) + \Re(z)\left( \log|z| - 1 \right)  - \frac12 \log |z| + \Ord\left(\frac1{|z|} \right).
    $$
This gives
\begin{align}
   \log\left( \frac{\Gamma(v)}{\Gamma(v+s)} \right) + & \log\left( \frac{\Gamma(\theta - v - s)}{\Gamma(\theta - v)} \right)
   = \nonumber \\
   & - \Im(v) \arg(v) + \Im(\theta - v) \arg(\theta - v) \label{eq:ratio-of-gammas-1-1} \\
   & + \Re(v) ( \log |v| - 1 ) - \Re(\theta - v) ( \log|\theta - v| -1 ) \label{eq:ratio-of-gammas-1-3} \\
   & + \Im(v + s) \arg(v + s) - \Im(\theta - v - s) \arg(\theta - v - s) \label{eq:ratio-of-gammas-1-2} \\
   & - \Re(v+s) ( \log |v + s| - 1 ) + \Re(\theta - v -s) ( \log |\theta - v - s| - 1 ) \label{eq:ratio-of-gammas-1-4} \\
   & - \frac{1}{2} \log \left| \frac{ v ( \theta - v -s) }{(v+s)(\theta -v )} \right|  + \Ord(|\theta|^{-1}) \label{eq:ratio-of-gammas-1-5}
\end{align}
since $|v|,|\theta -v|,|\theta - v- s|,|v+s| \geq c\theta$ for some constant $c>0$. Here and in the following, $a = \Ord(b)$ means that there is a constant $C$ such that $|a| \leq C b$ for all $\theta > \theta^*$. We will estimate the numbered terms in the above display one-by-one. 

We first estimate \Eqref{eq:ratio-of-gammas-1-3}: 
\begin{align*}
    \Re(v) ( \log |v| - 1 ) & - \Re(\theta - v) ( \log|\theta - v| -1 )  = \frac{\theta}{2} \log \frac{|v|}{|\theta - v|} - \cot(\varphi) |\Im(v)| \left( \log (|v||\theta -v|) -2 \right).
\end{align*}
Since the ratio inside the logarithm is $\Ord(1)$ for all $v \in \Cv{\varphi}$ we have for some $\varphi$-dependent constant $c$,
\begin{equation}
    \Eqref{eq:ratio-of-gammas-1-3} \leq \Ord(\theta) - c |\Im(v)|\log( |v| |\theta - v|).
    \label{eq:ratio-of-gammas-2-3}
\end{equation}
Thus, the terms in \Eqref{eq:ratio-of-gammas-1-3} dominate the terms in~\Eqref{eq:ratio-of-gammas-1-1} and this gives us the exponential decay in $v$ that we need.

Since $\Im(\theta - v - s) = - \Im(v+s)$,~\Eqref{eq:ratio-of-gammas-1-2}  becomes $\Im(v + s) ( \arg(v + s) + \arg(\theta - v - s) )$. From~\Eqref{eq:v on the contour Cv} and~\Eqref{eq:s on the contour Cs}, we get $\theta/2 + \delta = \Re(v+s) \geq \Re(\theta - (v+s)) = \theta/2 - \delta$. It follows that $\Im(v+s)$ and $\arg(v+s) + \arg(\theta - v -s)$ have opposite signs, and hence
\begin{equation}
    \eqref{eq:ratio-of-gammas-1-2}  \leq \Im(v + s) ( \arg(v + s) + \arg(\theta - v - s) ) \leq 0.
    \label{eq:ratio-of-gammas-2-2}
\end{equation}

For some constant $c > 0$,
\begin{equation}
    \begin{aligned}
        \Eqref{eq:ratio-of-gammas-1-4}
        & = - \frac{\theta}{2} \log \frac{|v + s|}{| \theta - v - s|} - \delta \log |v + s| | \theta - v - s| \\
        & \leq - \frac{\theta}{2} c - \delta \log\left( \frac{\theta}{2} - \delta \right) \left( \frac{\theta}{2} + \delta \right)\\
        & = \Ord(\theta) + \Ord(|\log(\theta)|).
    \end{aligned}
    \label{eq:ratio-of-gammas-2-4}
\end{equation}

The term \eqref{eq:ratio-of-gammas-1-5} is split into two terms: the first term $-\log|v/(\theta - v)|$ is $\Ord(|\log \theta|)$ for small $v$, and $\Ord(1)$ for $|v| \geq c \theta$. The second term $-\log |(\theta - v - s)/(v+s)|$ behaves similarly, and we get
\begin{equation}
    \eqref{eq:ratio-of-gammas-1-5} = \Ord(1) +  \Ord(| \log \theta|).
    \label{eq:ratio-of-gammas-2-5}
\end{equation}

From~\Eqref{eq:v on the contour Cv} and~\Eqref{eq:s on the contour Cs}, we get $|v + s - v'|^{-1} \leq \frac{2}{\theta}$. Finally, to analyze $\exp(\tau(s v + s^2/2))$, we look at the real part of $s v + s^2/2$:
\begin{equation}
    \begin{aligned}
        \Re(s v & + s^2/2)  \\
        & = \Re(s) \Re(v) - \Im(s) \Im(v)  + \frac{\Re(s)^2 - \Im(s)^2}{2}\\
        & = \Re(s) \Re(v) + \frac{\Re(s)^2}{2} + \frac{\Im(v)^2}{2} - \frac{(\Im(s) + \Im(v))^2}{2} \\
        & = (\delta + \cot(\varphi) |\Im(v)|) (\frac{\theta}{2} - \cot(\varphi) |\Im(v)|) + \frac{(\delta + \cot(\varphi)|\Im(v)|)^2}{2} + \frac{\Im(v)^2}{2} - \frac{(\Im(s) + \Im(v))^2}{2}  \\
        & = \frac{\theta \delta + \delta^2}{2} + \cot(\varphi) \frac{\theta}{2} |\Im(v)|  + (1 - \cot(\varphi)^2)\frac{\Im(v)^2}{2} -  \frac{(\Im(s) + \Im(v))^2}{2}  \\
        & \leq C\theta + \cot(\varphi) \frac{\theta}{2} |\Im(v)| ,
    \end{aligned}
    \label{eq:ratio-of-gammas-2-7}
\end{equation}
using $\cot(\varphi) \geq 1$.

Putting~\Eqref{eq:ratio-of-gammas-2-3},~\Eqref{eq:ratio-of-gammas-2-2},~\Eqref{eq:ratio-of-gammas-2-4},~\Eqref{eq:ratio-of-gammas-2-5}, and \Eqref{eq:ratio-of-gammas-2-7} together with 
$|\sin(\pi s)|^{-1} \leq Ce^{-\pi |\Im(s)|}$,  we get
\begin{equation}
    I_{u,\tau}(v,v',s) \leq e^{C(\theta,N,\tau)} e^{ - N |\Im(v)|(\log |v| - c \tau \cot(\varphi) \theta)} e^{-\pi |\Im(s)|}
    \label{eq:estimate on integrand of BCFV kernel along vertical line}
\end{equation}
where
\[
    C(\theta,N,\tau) = \Ord( N (\theta + |\log \theta| + \theta^{-1} + 1) + \tau \theta).
\]
Integrating this over $s$ completes the proof.
\end{proof}

\bibliographystyle{plainnat}
\bibliography{master-bibtex}

\end{document}